\documentclass{amsart}
\numberwithin{equation}{section}
\usepackage[english]{babel}
\usepackage[T1]{fontenc}
\usepackage[latin1]{inputenc}
\usepackage{indentfirst}
\usepackage{enumitem}
\usepackage{amsmath,amssymb, amsbsy}
\usepackage{amsfonts}
\usepackage{hyperref}
\usepackage{latexsym}
\usepackage{amsthm}
\usepackage[dvips]{graphicx}
\DeclareGraphicsExtensions{.pdf,.png,.jpg,.eps}

\usepackage{color}
\usepackage[latin1]{inputenc}
\usepackage[active]{srcltx}
\definecolor{Arancio}{cmyk}{0,0.61,0.87,0}

\definecolor{blus}{RGB}{0,102,204}

\newcommand{\brd}[1]{\mathbb{#1}}
\newcommand{\R}{\brd{R}}
\newcommand{\abs}[1]{\left\lvert {#1} \right\rvert}

\newcommand\ddfrac[2]{\frac{\displaystyle #1}{\displaystyle #2}}
\newtheorem{teo}{Theorem}[section]
\newtheorem{Corollary}[teo]{Corollary}
\newtheorem{Lemma}[teo]{Lemma}
\newtheorem{Theorem}[teo]{Theorem}
\newtheorem{Proposition}[teo]{Proposition}
\theoremstyle{definition}

\newtheorem{remark}[teo]{Remark}

\def\avint{\mathop{\,\rlap{-}\!\!\int}\nolimits}

\begin{document}

\title[Boundary regularity estimates in H\"older spaces with variable exponent]
{Boundary regularity estimates in H\"older spaces with variable exponent}
\date{\today}

\author{Stefano Vita}

\address[S. Vita]{Dipartimento di Matematica
\newline\indent
Politecnico di Milano
\newline\indent
Via Andrea Maria Amp\`ere 2, 20133, Milano, Italy}
\email{stefano.vita@polimi.it}


\thanks{{\it Acknowledgment:} The author is a research fellow of Istituto Nazionale di Alta Matematica INDAM}

\maketitle


\begin{abstract}
We present a general blow-up technique to obtain local regularity estimates for solutions, and their derivatives, of second order elliptic equations in divergence form in H\"older spaces with variable exponent. The procedure allows to extend the estimates up to a portion of the boundary where Dirichlet or Neumann boundary conditions are prescribed and produces a Schauder theory for partial derivatives of solutions of any order $k\in\mathbb N$. The strategy relies on the construction of a class of suitable regularizing problems and an approximation argument. The estimates we obtain are sharp with respect to the regularity or integrability conditions on variable coefficients, boundaries, boundary data and right hand sides respectively in H\"older and Lebesgue spaces, both with variable exponent.
\end{abstract}


\section{Introduction}
Given a bounded domain $\Omega$ in $\R^n$ with locally $C^1$ boundary around $0\in\partial\Omega$, $n\geq2$, let us consider the following second order uniformly elliptic problem in divergence form
\begin{equation}\label{eq2}
\begin{cases}
-\mathrm{div}\left(A\nabla u\right)=f+\mathrm{div}F+Vu+b\cdot\nabla u &\mathrm{in \ } B_r\cap\Omega\\
u=g\quad\mathrm{or}\quad A\nabla u\cdot\nu=h &\mathrm{on \ } B_r\cap\partial\Omega,
\end{cases}
\end{equation}
where $B_r=\{x\in\R^n \ : \ |x|<r\}$ and the variable coefficients matrix $A(x)=(a_{ij}(x))_{i,j=1}^{n}$ is continuous, symmetric and uniformly elliptic, i.e.
$$\lambda |\xi|^2\leq A(x)\xi\cdot\xi\leq \Lambda |\xi|^2\qquad\mathrm{for \ given \ }0<\lambda\leq\Lambda<+\infty.$$
We are interested in local regularity estimates up to the boundary in variable exponent H\"older spaces for solutions to \eqref{eq2} and their derivatives which are sharp with respect to integrability or regularity conditions on given data and on the boundary $\partial\Omega$, respectively in variable exponent Lebesgue and H\"older spaces.
%

Regularity theory for linear partial differential equations in variable exponent spaces comprehends many contributions. Just to cite a few, we mention two seminal papers by Diening, R\r{u}\v{z}i\v{c}ka and collaborators \cite{DiePoisson,DieCrelle} where the authors deal with regularity in variable exponent Sobolev spaces for the divergence, the Poisson and the Stokes problems (see also \cite{DieRuz1,DieRuz2}). Local regularity estimates for solutions to second order uniformly elliptic equations in H\"older spaces with variable exponent were obtained in \cite{BieGor1,BieGor2} both for operators in non-divergence and in divergence form. We also report \cite{Bie}, where the author deals with the parabolic case. Moreover, there exists an extensive production on regularity results for nonlinear equations with non standard growth (the model operator is the $p(\cdot)$-Laplacian); just to name a few we refer to \cite{AceMin1,AceMin2}.

Our target is twofold: on one side, we introduce a blow-up procedure which allows to obtain estimates in variable exponent H\"older spaces for solutions to second order linear equations; on the other, we are able to cover the results in \cite{BieGor2} with very different techniques and extend the regularity estimates up to a piece of the boundary where Dirichlet or Neumann data are prescribed, dealing additionally with transport terms. Eventually, we show how to iterate our gradient estimates in order to provide a complete Schauder theory for derivatives of solutions of any order $k\in\mathbb{N}$. Our approach uses very few facts about variable exponent spaces and relies mostly on a perturbation argument and an approximation scheme.

\subsection{Variable exponent spaces}
Although the theory of variable exponent spaces has been developed mostly in the last twenty years, the related literature is very extensive and we will not be able to list all the known properties of these spaces, but we will recall from time to time only some properties that will be needed throughout the proofs. We invite the reader who is interested in deepening the knowledge of variable exponent spaces to the reading of the monograph \cite{Die}, which collects a substantial part of the theory on this subject.
For an accurate definition of the variable exponent Lebesgue space $L^{p(\cdot)}(\Omega)$ we refer to \cite[Chapter 3]{Die}. Briefly, a measurable function $p:\Omega\to[1,+\infty)$ is called exponent. Let $\underline p=\inf_\Omega p(x)$ and $\overline p=\sup_\Omega p(x)$. If $\overline p<+\infty$ then $p$ is called bounded exponent. For a bounded exponent $p$, the variable exponent Lebesgue space $L^{p(\cdot)}(\Omega)$ consists of measurable functions $f:\Omega\to\R$ such that the modular
\begin{equation*}
\rho_\Omega^p(f):=\int_{\Omega}|f(x)|^{p(x)}\mathrm{d}x
\end{equation*}
is finite. The norm is defined as
\begin{equation*}
\|f\|_{L^{p(\cdot)}(\Omega)}=\inf\{\lambda>0 \ : \ \rho_\Omega^p(f/\lambda)\leq1\}.
\end{equation*}
A precise definition of the variable exponent H\"older space $C^{0,\alpha(\cdot)}(\Omega)$ can be found in \cite{AlmSam0,AlmSam}. Let $\alpha:\Omega\subseteq\R^n\to\R$ be a measurable function such that $0<\underline\alpha\leq\alpha(x)\leq\overline\alpha\leq1$. The variable exponent H\"older space $C^{0,\alpha(\cdot)}(\Omega)$ consists of continuous functions such that
\begin{equation*}
\|u\|_{C^{0,\alpha(\cdot)}(\Omega)}=\|u\|_{L^\infty(\Omega)}+[u]_{C^{0,\alpha(\cdot)}(\Omega)}
\end{equation*}
is finite, where the variable H\"older seminorm is defined as
\begin{equation}\label{holderseminorm1}
[u]_{C^{0,\alpha(\cdot)}(\Omega)}=\sup_{\substack{x,y\in\Omega \\ 0<|x-y|\leq1}}\frac{|u(x)-u(y)|}{|x-y|^{\alpha(x)}}=\sup_{\substack{x,y\in\Omega \\ 0<|x-y|\leq1}}\frac{|u(x)-u(y)|}{|x-y|^{\max\{\alpha(x),\alpha(y)\}}}.
\end{equation}
We remark here that since we are interested in local regularity, without loss of generality, we can always localize our problems in small balls $B_r$ of radius $r\leq 1/2$. In this way, the condition of \textsl{closeness of points} $|x-y|\leq1$ in the definition of the H\"older seminorm in \eqref{holderseminorm1} becomes natural and can be avoided.
In general, for $k\in\mathbb N$ we say that a $C^k(\Omega)$ function $u$ belongs to $C^{k,\alpha(\cdot)}(\Omega)$ if
\begin{equation*}
\|u\|_{C^{k,\alpha(\cdot)}(\Omega)}=\sum_{i=0}^k\sum_{|\beta_i|=i}\|D^{\beta_i}u\|_{L^\infty(\Omega)}+[u]_{C^{k,\alpha(\cdot)}(\Omega)}
\end{equation*}
is finite, where
\begin{equation*}
[u]_{C^{k,\alpha(\cdot)}(\Omega)}=\sup_{\substack{x,y\in\Omega \\ 0<|x-y|\leq1\\ |\beta_k|=k}}\frac{|D^{\beta_k} u(x)-D^{\beta_k} u(y)|}{|x-y|^{\alpha(x)}}
\end{equation*}
and $D^{\beta_i}u$ is a partial derivative of $u$ of order $i = |\beta_i|$, with $\beta_i\in\mathbb N^n$ multiindex.

A very important condition on exponents which ensures many properties of variable exponent spaces is the $\log$-H\"older continuity. A continuous and bounded function $u:\Omega\to\R$ is said to be $\log$-H\"older continuous if there exists a positive constant such that for any $x,y\in\Omega$ with $x\neq y$
\begin{equation}\label{log}
|u(x)-u(y)|\leq \frac{c}{\log\left(e+\frac{1}{|x-y|}\right)}.
\end{equation}
We will denote the space of $\log$-H\"older continuous functions by $C^{0,1/|\log\cdot|}(\Omega)$. Then, the following are the families of exponents that we are going to consider
\begin{equation*}
\mathcal{P}^{\log}(\Omega)=\left\{p\in C^{0,1/|\log\cdot|}(\Omega) \ : \ 1<\underline p\leq p(x)\leq\overline p<+\infty\right\}
\end{equation*}
and
\begin{equation*}
\mathcal{A}^{\log}(\Omega)=\left\{\alpha\in C^{0,1/|\log\cdot|}(\Omega) \ : \ 0<\underline\alpha\leq\alpha(x)\leq\overline\alpha<1\right\}.
\end{equation*}
We will often indicate the $\log$-H\"older modulus of continuity in \eqref{log}, which is monotone increasing, by the equivalent (for small $0<t<<1$)
\begin{equation*}
\omega(t)=\frac{1}{|\log t|}.
\end{equation*}

\subsection{Structure of the paper and main results}
Throughout the paper, our strategy is the following: we first localize the problem at a boundary point lying on $\partial\Omega$; that is, we consider problem \eqref{eq2}. We can imagine, up to a rotation, that the portion of the boundary $\partial\Omega\cap B_r$ where boundary conditions are prescribed can be described as the graph of a function $\varphi$ and the portion of the domain where the equation is satisfied $\Omega\cap B_r$ as its epigraph. In other words, we rewrite \eqref{eq2} as
\begin{equation}\label{boundABC}
\begin{cases}
-\mathrm{div}\left(A\nabla u\right)=f+\mathrm{div}F+Vu+b\cdot\nabla u &\mathrm{in \ } B_r\cap\{x_n>\varphi(x')\}\\
u=g \quad \mathrm{or} \quad A\nabla u\cdot\nu=h  &\mathrm{on \ } B_r\cap\{x_n=\varphi(x')\},
\end{cases}
\end{equation}
with $\varphi\in C^1(B'_r)$, $\varphi(0)=0$ and where $x=(x',x_n)\in\R^{n-1}\times\R$, $B'_r=B_r\cap\{x_n=0\}$. Hence, regularity of the boundary $\partial\Omega$ must be understood as regularity of the function $\varphi$. We would like here to leave the boundary free to enjoy a regularity condition with variable exponent $\partial\Omega\in C^{k,\alpha'(\cdot)}$ (in \cite{BieGor1} the boundary satisfies a classic regularity condition $\partial\Omega\in C^{k,\overline\alpha}$ with $\overline\alpha=\sup\alpha$). We set here another notation: from now on, we denote a variable exponent which depends only on the first $n-1$ variables by $\alpha'$, and
\begin{equation}\label{tildealpha}
\tilde\alpha(x',x_n)=\alpha'(x')
\end{equation}
stands for the constant extension in the $n^{\mathrm{th}}$-variable of the exponent $\alpha'$. Then, we define a local diffeomorphism in order to straighten the boundary of $\Omega$. We end up with a problem on a half ball
\begin{equation}\label{eq1}
\begin{cases}
-\mathrm{div}\left(A\nabla u\right)=f+\mathrm{div}F+Vu+b\cdot\nabla u &\mathrm{in \ } B_r^+=B_r\cap\{x_n>0\}\\
u=g \quad\mathrm{or}\quad A\nabla u\cdot\nu=h &\mathrm{in \ } B'_r,
\end{cases}
\end{equation}
where the new solution, right hand sides and boundary data are related to the original ones by composition with the diffeomorphism and the new variable coefficients matrix is a suitable product of the old one with Jacobian matrixes then composed with the diffeomorphism too. We show that the problem does not change qualitatively after the diffeomorphism, i.e. the belonging to variable exponent spaces is preserved in an appropriate sense. Hence, we regularize problem \eqref{eq1}, by smoothing coefficients, forcing terms and boundary data.
Then, the regularized problems enjoy local boundary regularity estimates by classical results, and we prove by a contradiction argument, which relies on a blow-up procedure and a Liouville type theorem, that the constants in the estimates are uniform with respect to the parameter of regularization. Hence, we construct a scheme of approximation which brings the same estimates to any weak solution of \eqref{eq1} and ensures sharp regularity; that is, we prove the following results, which are H\"older and gradient estimates up to the flat boundary.

\begin{Theorem}\label{teo1}
Let $r>0$, $p,q,m_1,m_2\in\mathcal{P}^{\log}(B_r^+)$ with $\underline p,\underline m_1>\frac{n}{2}$, $\underline q,\underline m_2>n$. Let $s\in\mathcal{P}^{\log}(B'_r)$ with $\underline s>n-1$. Let $\alpha\in \mathcal{A}^{\log}(B_r^+)$ with
\begin{equation}\label{alpha1}
\alpha(x)\leq \min\left\{2-\frac{n}{p(x)},2-\frac{n}{m_1(x)},1-\frac{n}{q(x)}\right\}, \qquad \alpha(x',0)\leq 1-\frac{n-1}{s(x')},
\end{equation}
and $A\in C(B_r^+\cup B'_r)$. Then there exists a positive constant (depending on $r,p,q,m_1,m_2,\alpha,n$ and $\|V\|_{L^{m_1(\cdot)}},\|b\|_{L^{m_2(\cdot)}}$) such that for weak solutions to \eqref{eq1} holds
\begin{equation}\label{C0alphaest}
\|u\|_{C^{0,\alpha(\cdot)}(B_{r/2}^+)}\leq c\left(\|u\|_{L^\infty(B_r^+)}+\|f\|_{L^{p(\cdot)}(B_r^+)}+\|F\|_{L^{q(\cdot)}(B_r^+)}+\begin{cases}\|g\|_{C^{0,\alpha(\cdot)}(B'_r)}\quad\mathrm{or}\\
\|h\|_{L^{s(\cdot)}(B'_r)}\end{cases}\right).
\end{equation}
\end{Theorem}

\begin{Theorem}\label{teo2}
Let $r>0$, $p,m_1,m_2\in\mathcal{P}^{\log}(B_r^+)$ with $\underline p,\underline m_1,\underline m_2>n$. Let $\alpha\in \mathcal{A}^{\log}(B_r^+)$ with
\begin{equation}\label{alpha2}
\alpha(x)\leq \min\left\{1-\frac{n}{p(x)},1-\frac{n}{m_1(x)},1-\frac{n}{m_2(x)}\right\}
\end{equation}
and $A\in C^{0,\alpha(\cdot)}(B_r^+)$. Then there exists a positive constant (depending on $r,p,m_1,m_2,\alpha,n$ and $\|V\|_{L^{m_1(\cdot)}},\|b\|_{L^{m_2(\cdot)}}$) such that for weak solutions to \eqref{eq1} holds
\begin{equation}\label{C1alphaest}
\|u\|_{C^{1,\alpha(\cdot)}(B_{r/2}^+)}\leq c\left(\|u\|_{L^\infty(B_r^+)}+\|f\|_{L^{p(\cdot)}(B_r^+)}+\|F\|_{C^{0,\alpha(\cdot)}(B_r^+)}+\begin{cases}\|g\|_{C^{1,\alpha(\cdot)}(B'_r)}\quad\mathrm{or}\\
\|h\|_{C^{0,\alpha(\cdot)}(B'_r)}\end{cases}\right).
\end{equation}
\end{Theorem}
\begin{remark}
Let us stress the fact that the estimates stated above are truly boundary estimates in the sense that the structure of the variable coefficients matrix $A$ obtained after the diffeomorphism does not allow in general, even in the case of homogeneous boundary conditions, a standard even or odd reflection of the equation across $\{x_n=0\}$ which preserves continuity of coefficients. In other words, the hyperplane $\{x_n=0\}$ is not invariant with respect to $A$ when $x_n=0$.
\end{remark}

Then, we show that the information obtained on the straightened problem immediately translates into boundary regularity for the original curved problem.

\begin{Corollary}\label{cor1}
Let $r>0$, $\alpha,\tilde\alpha\in \mathcal{A}^{\log}(B_r\cap\Omega)$ and $u$ be a weak solution to \eqref{eq2}. Then
\begin{itemize}
\item[i)] if $p,q,m_1,m_2\in \mathcal{P}^{\log}(B_r\cap\Omega)$ with $\underline p,\underline m_1>\frac{n}{2}$, $\underline q,\underline m_2>n$, $s\in \mathcal{P}^{\log}(B_r\cap\partial\Omega)$ with $\underline s>n-1$, $\alpha$ satisfies
\begin{equation*}\label{alpha3}
\alpha(x)\leq \min\left\{2-\frac{n}{p(x)},2-\frac{n}{m_1(x)},1-\frac{n}{q(x)}\right\} \quad\mathrm{in \ }B_r\cap\Omega,
\end{equation*}
\begin{equation*}\label{alpha31}
\alpha(x)\leq 1-\frac{n-1}{s(x)}\quad\mathrm{in \ }B_r\cap\partial\Omega,
\end{equation*}
$\partial\Omega\in C^{1}$, $A\in C(B_r\cap\overline\Omega)$, $f\in L^{p(\cdot)}(B_r\cap\Omega)$, $F\in L^{q(\cdot)}(B_r\cap\Omega)$, $V\in L^{m_1(\cdot)}(B_r\cap\Omega)$, $b\in L^{m_2(\cdot)}(B_r\cap\Omega)$ and $g\in C^{0,\alpha(\cdot)}(B_r\cap\partial\Omega)$ or $h\in L^{s(\cdot)}(B_r\cap\partial\Omega)$, then $u\in C^{0,\alpha(\cdot)}_{\mathrm{loc}}(B_r\cap\overline\Omega)$;
\item[ii)] if $p,m_1,m_2\in \mathcal{P}^{\log}(B_r\cap\Omega)$ with $\underline p,\underline m_1,\underline m_2>n$, $\alpha$ satisfies
\begin{equation*}\label{alpha4}
\alpha(x)\leq \min\left\{1-\frac{n}{p(x)},1-\frac{n}{m_1(x)},1-\frac{n}{m_2(x)}\right\}\quad\mathrm{in \ }B_r\cap\Omega,
\end{equation*}
$\partial\Omega\in C^{1,\alpha'(\cdot)}$, $\tilde\alpha$ satisfies \eqref{tildealpha}, $\gamma=\min\{\alpha,\tilde\alpha\}$, $A\in C^{0,\alpha(\cdot)}(B_r\cap\Omega)$, $f\in L^{p(\cdot)}(B_r\cap\Omega)$, $F\in C^{0,\alpha(\cdot)}(B_r\cap\Omega)$, $V\in L^{m_1(\cdot)}(B_r\cap\Omega)$, $b\in L^{m_2(\cdot)}(B_r\cap\Omega)$ and $g\in C^{1,\alpha(\cdot)}(B_r\cap\partial\Omega)$ or $h\in C^{0,\alpha(\cdot)}(B_r\cap\partial\Omega)$, then $u\in C^{1,\gamma(\cdot)}_{\mathrm{loc}}(B_r\cap\overline\Omega)$.
\end{itemize}
\end{Corollary}


Eventually, by an inductive argument, the gradient estimate in Theorem \ref{teo2} implies the following Schauder regularity for partial derivatives of solutions of any order $k\in\mathbb{N}$.


\begin{Corollary}\label{cor3}
Let $r>0$, $k\geq0$, $\alpha,\tilde\alpha\in \mathcal{A}^{\log}(B_r\cap\Omega)$, and $u$ be a weak solution to

\begin{equation*}
\begin{cases}
-\mathrm{div}\left(A\nabla u\right)=\mathrm{div}F &\mathrm{in \ } B_r\cap\Omega\\
u=g\quad\mathrm{or}\quad A\nabla u\cdot\nu=h &\mathrm{on \ } B_r\cap\partial\Omega,
\end{cases}
\end{equation*}
with $A,F\in C^{k,\alpha(\cdot)}(B_r\cap\Omega)$, $g\in C^{k+1,\alpha(\cdot)}(B_r\cap\partial\Omega)$ or $h\in C^{k,\alpha(\cdot)}(B_r\cap\partial\Omega)$, $\partial\Omega\in C^{k+1,\alpha'(\cdot)}$, $\tilde\alpha$ satisfies \eqref{tildealpha}, $\gamma=\min\{\alpha,\tilde\alpha\}$. Then $u\in C^{k+1,\gamma(\cdot)}_{\mathrm{loc}}(B_r\cap\overline\Omega)$.
\end{Corollary}

\begin{remark}
We would like to stress the fact that all the results stated above at boundary points (either on curved domains or half balls with a boundary condition on the flat boundary) continue to hold a fortiori in the interior (on balls inside the domain).
\end{remark}


\section{Uniform H\"older and gradient estimates for regularized problems}
In this section, we apply a local diffeomorphism on \eqref{eq2} and we associate to the new straightened problem a family of regularized equations by smoothening coefficients, forcing terms and boundary data. Hence, we prove H\"older and gradient estimates with variable exponent for the regularized problems which are uniform with respect to the regularization parameter.

\subsection{A diffeomorphism straightening the boundary}\label{sectdiffeo}

Let us consider the following local diffeomorphism
\begin{equation}\label{diffeomorphism}
\psi(x',x_n)=(x',x_n+\varphi(x')),
\end{equation}
Hence, there exists a small enough $R>0$ such that
$$\psi \ : B_R\cap\{x_n>0\}\mapsto B_R\cap\{x_n>\varphi(x')\},$$
mapping the boundary $B_R\cap\{x_n=0\}$ into the boundary $B_R\cap\{x_n=\varphi(x')\}$.
The Jacobian associated with $\psi$ is given by
\begin{align*}
J_\psi(x')=\left(\begin{array}{c|c}
\mathbb{I}_{n-1}&{\mathbf 0}\\\hline
(\nabla \varphi(x'))^T&1
\end{array}\right), \qquad \mathrm{with}\quad |\mathrm{det} \, J_\psi(x')|\equiv1.
\end{align*}
Up to a possible dilation, one can translate the study of \eqref{boundABC} into the study of the following problem for $v=u\circ\psi$
\begin{equation*}\label{boundM}
\begin{cases}
-\mathrm{div}\left(M\nabla v\right)=\tilde f+\mathrm{div}\tilde F+\tilde Vv+\tilde b\cdot\nabla v &\mathrm{in \ }B_r^+= B_r\cap\{x_n>0\}\\
v=\tilde g \quad\mathrm{or}\quad M\nabla v\cdot\nu=\tilde h&\mathrm{on \ }B'_r=B_r\cap\{x_n=0\},
\end{cases}
\end{equation*}
where the new matrix $M$ is given by
\begin{equation*}
M=(J_\psi^{-1}) (A \circ\psi) (J_\psi^{-1})^T,
\end{equation*}
and
\begin{equation*}
\tilde f=f\circ\psi,\qquad \tilde F=F\circ\psi, \qquad \tilde V=V\circ\psi\qquad \tilde b=b\circ\psi\qquad\tilde g=g\circ\psi,\qquad \tilde h=h\circ\psi.
\end{equation*}
Now we are going to state some easy preliminary lemmas about product of variable exponent H\"older continuous functions and which show the precise translation of integrability and regularity conditions after the composition with a local diffeomorphism. Although some results have an intersection with \cite[Lemma 4.4]{BieGor1}, for completeness we repeat them here in a version which will be suitable for our purposes.

\begin{Lemma}
Let be $k\in\mathbb N$, $\Omega$ a bounded domain in $\R^n$ and let be $u\in C^{k,\alpha(\cdot)}(\Omega)$ and $v\in C^{k,\beta(\cdot)}(\Omega)$ with $0<\underline\alpha\leq\alpha(x)\leq\overline\alpha\leq1$ and $0<\underline\beta\leq\beta(x)\leq\overline\beta\leq1$. Then the product $uv\in C^{k,\min\{\alpha(\cdot),\beta(\cdot)\}}(\Omega)$.
\end{Lemma}
\proof
Let us procede by induction and let us start with the case $k=0$. Let $x,y\in\Omega$ with $0<|x-y|\leq1$. Then
\begin{eqnarray*}
|uv(x)-uv(y)|&\leq&|v(y)(u(x)-u(y))|+|u(x)(v(x)-v(y))|\\
&\leq&\sup_{x\in\Omega}|v(x)|c_\alpha|x-y|^{\max\{\alpha(x),\alpha(y)\}}+\sup_{x\in\Omega}|u(x)|c_\beta|x-y|^{\max\{\beta(x),\beta(y)\}}\\
&\leq&C|x-y|^{\min\{\max\{\alpha(x),\alpha(y)\},\max\{\beta(x),\beta(y)\}\}}\\
&\leq&C|x-y|^{\max\{\min\{\alpha(x),\beta(x)\},\min\{\alpha(y),\beta(y)\}\}}.
\end{eqnarray*}
Hence
\begin{equation*}
[uv]_{C^{0,\min\{\alpha(\cdot),\beta(\cdot)\}}(\Omega)}=\sup_{\substack{x,y\in\Omega \\ 0<|x-y|\leq1}}\frac{|uv(x)-uv(y)|}{|x-y|^{\max\{\min\{\alpha(x),\beta(x)\},\min\{\alpha(y),\beta(y)\}\}}}<+\infty.
\end{equation*}
Then let us suppose the result true in the generic case $k$ and let us prove the result in case $k+1$. We are assuming that $u\in C^{k,\alpha(\cdot)}(\Omega)$ and $v\in C^{k,\beta(\cdot)}(\Omega)$ implies $uv\in C^{k,\min\{\alpha(\cdot),\beta(\cdot)\}}(\Omega)$. Let us suppose now that $u\in C^{k+1,\alpha(\cdot)}(\Omega)$ and $v\in C^{k+1,\beta(\cdot)}(\Omega)$. Taken a partial derivative $D^{\alpha_{k+1}}(uv)$ for $|\alpha_{k+1}|=k+1$, then there exists $i\in\{1,...,n\}$ and multiindex $\alpha_k$ with $|\alpha_{k}|=k$, such that
$$D^{\alpha_{k+1}}(uv)=D^{\alpha_k}(\partial_i(uv))=D^{\alpha_k}(v\partial_iu)+D^{\alpha_k}(u\partial_iv).$$
By inductive hypothesis $v\partial_iu\in C^{k,\alpha(\cdot)}(\Omega)$ and $u\partial_iv\in C^{k,\beta(\cdot)}(\Omega)$. So, $D^{\alpha_{k+1}}(uv)$ is the sum of two functions belonging respectively to $C^{0,\alpha(\cdot)}(\Omega)$ and $C^{0,\beta(\cdot)}(\Omega)$. By trivial inclusion of spaces, we get that $D^{\alpha_{k+1}}(uv)\in C^{0,\min\{\alpha(\cdot),\beta(\cdot)\}}(\Omega)$ which proves the result.
\endproof

\begin{Lemma}
If $\alpha,\beta\in \mathcal{A}^{\mathrm{log}}(\Omega)$, then the pointwise minimum $\gamma=\min\{\alpha,\beta\}\in \mathcal{A}^{\mathrm{log}}(\Omega)$.
\end{Lemma}
\proof
Let $x,y\in \Omega$ with $x\neq y$, and let us consider $|\gamma(x)-\gamma(y)|$. If $\gamma(x)=\alpha(x)$ and $\gamma(y)=\alpha(y)$ or $\gamma(x)=\beta(x)$ and $\gamma(y)=\beta(y)$ then we can conclude by $\log$-H\"older continuity of $\alpha$ or $\beta$ that
\begin{equation*}
|\gamma(x)-\gamma(y)|\leq c\frac{1}{\log\left(e+\frac{1}{|x-y|}\right)}.
\end{equation*}
Hence we can assume without loss of generality that $\gamma(x)=\alpha(x)$ and $\gamma(y)=\beta(y)$, which implies that on the segment $[x,y]$ the graphs of the one dimensional restrictions of $\alpha$ and $\beta$ (which are continuous) have to cross each other by the intermediate zero theorem; that is, there exists an intermediate point $z$ such that $\alpha(z)=\beta(z)$ and $\omega(|x-z|),\omega(|y-z|)$ are both less or equal than $\omega(|x-y|)$ by the non decreasing monotonicity of the modulus of continuity $\omega(t)=1/\log(e+1/t)$. Hence,
\begin{equation*}
|\gamma(x)-\gamma(y)|=|\alpha(x)-\beta(y)|\leq|\alpha(x)-\alpha(z)|+|\beta(z)-\beta(y)|\leq c\frac{1}{\log\left(e+\frac{1}{|x-y|}\right)}.
\end{equation*}
\endproof

\begin{Lemma}
Let $\psi:\R^n\to\R^n$ be a local diffeomorphism between $\Omega'=\psi^{-1}\Omega$ and $\Omega$ such that $\psi^{-1}$ is Lipschitz continuous in $\Omega$. Then
\begin{itemize}
\item[$i)$] if $f\in L^{p(\cdot)}(\Omega)$, then $f\circ\psi\in L^{p\circ\psi(\cdot)}(\Omega')$;
\item[$ii)$] if $\psi$ is bi-Lipschitz continuous and $\mathrm{diam}(\Omega),\mathrm{diam}(\Omega')\leq 1$, then $u\in C^{0,\alpha(\cdot)}(\Omega)$ if and only if $u\circ\psi\in C^{0,\alpha\circ\psi(\cdot)}(\Omega')$.
\end{itemize}
\end{Lemma}
\proof
$i)$ Let us compute the modular of $f\circ\psi$
\begin{eqnarray*}
\rho_{\Omega'}^{p\circ\psi}(f\circ\psi)&=&\int_{\Omega'}|f\circ\psi(y)|^{p\circ\psi(y)}\mathrm{d}y\\
&=&\int_\Omega|f\circ\psi(\psi^{-1}(x))|^{p\circ\psi(\psi^{-1}(x))}|\mathrm{det} \, J_{\psi^{-1}}(x)|\mathrm{d}x\\
&\leq&\sup_{\Omega}|\mathrm{det} \, J_{\psi^{-1}}| \, \rho_{\Omega}^{p}(f)<+\infty.
\end{eqnarray*}
$ii)$ Let us consider the variable exponent H\"older seminorm of $u\circ\psi$
\begin{eqnarray*}
[u\circ\psi]_{C^{0,\alpha\circ\psi(\cdot)}(\Omega')}&=&\sup_{\substack{x,y\in\Omega' \\ x\neq y}}\frac{|u\circ\psi(x)-u\circ\psi(y)|}{|x-y|^{\alpha\circ\psi(x)}}\\
&=&\sup_{\substack{\overline x,\overline y\in\Omega \\ \psi^{-1}(\overline x)\neq\psi^{-1}(\overline y)}}\frac{|u(\overline x)-u(\overline y)|}{|\psi^{-1}(\overline x)-\psi^{-1}(\overline y)|^{\alpha(\overline x)}}\\
&\leq&C\sup_{\substack{\overline x,\overline y\in\Omega \\ \overline x\neq\overline y}}\frac{|u(\overline x)-u(\overline y)|}{|\overline x-\overline y|^{\alpha(\overline x)}}=C[u]_{C^{0,\alpha(\cdot)}(\Omega)},
\end{eqnarray*}
where in the last inequality above we have used the Lipschitz continuity and bijectivity of $\psi$. We remark that the condition on diameters of $\Omega,\Omega'$ allows us to avoid the condition of \textsl{closeness of points} $|x-y|\leq1$ in the definition of the seminorm. Then, using the Lipschitz continuity of $\psi^{-1}$ we can obtain the other implication.
\endproof

We would like to remark here that our local diffeomorphism $\psi$ in \eqref{diffeomorphism} can be set between two domains with diameter less or equal than $1$. Moreover $\psi$ is bi-Lipschitz continuous with $|\mathrm{det} \, J_\psi|=|\mathrm{det} \, J_{\psi^{-1}}|\equiv1$ (the Lipschitz constant is $L=1$) and hence $\psi$ is also an isometry between variable exponent Lebesgue/H\"older spaces and the related ones with exponents obtained after the composition with the diffeomorphism itself.

Nevertheless, if $\partial\Omega\in C^{k,\alpha'(\cdot)}$ for some integer $k\geq 1$, then $\psi^{-1}\in C^{k,\tilde\alpha(\cdot)}$ with $\tilde\alpha$ given in \eqref{tildealpha}. This means also that, if $A\in C^{k-1,\alpha(\cdot)}(B_r\cap\Omega)$ then $M\in C^{k-1,\gamma\circ\psi(\cdot)}(B_r^+)$, where $\gamma=\min\{\alpha,\tilde\alpha\}$.

Of course, $\log$-H\"older continuity of exponents is preserved after composition with $\psi$ ($p\in\mathcal P^{\log}(B_r\cap\Omega)\Longleftrightarrow p\circ\psi\in\mathcal P^{\log}(B_r^+)$ and $\alpha\in\mathcal A^{\log}(B_r\cap\Omega)\Longleftrightarrow\alpha\circ\psi\in\mathcal A^{\log}(B_r^+)$), and the same happens with pointwise inequalities between exponents, for example
\begin{equation*}
\alpha(x)\leq 2-\frac{n}{p(x)}\qquad \mathrm{in \ }B_r\cap\Omega\qquad\Longleftrightarrow\qquad\alpha(\psi(y))\leq 2-\frac{n}{p(\psi(y))}\qquad \mathrm{in \ }B_r^+.
\end{equation*}
Thanks to the lemmas and remarks stated in this section we are in position to claim the following proposition: Theorem \ref{teo1} and Theorem \ref{teo2} imply respectively Corollary \ref{cor1} part $i)$ and $ii)$.
\begin{proof}[Proof of "Theorem \ref{teo1} $\Rightarrow$ Corollary \ref{cor1} part $i)$" and of "Theorem \ref{teo2} $\Rightarrow$ Corollary \ref{cor1} part $ii)$"]
Now we are going to show how to translate boundary regularity estimates for the straightened problem into boundary local regularity for the original curved problem:
\begin{itemize}
\item[i)] when the coefficients matrix $A$ is continuous and the boundary $\partial\Omega$ is locally $C^1$, then $\psi^{-1}\in C^1$ and the coefficients of $M$ are continuous. Therefore, if we obtain $C^{0,\alpha\circ\psi(\cdot)}_{\mathrm{loc}}(B_r^+\cup B'_r)$-regularity for $v$, then it translates into $C^{0,\alpha(\cdot)}_{\mathrm{loc}}(B_r\cap\overline\Omega)$-regularity for $u$ through composition with $\psi^{-1}\in C^1$. In other words, Theorem \ref{teo1} implies Corollary \ref{cor1} part $i)$;
\item[ii)] when the coefficients matrix $A$ is $\alpha(\cdot)$-H\"older continuous and the boundary $\partial\Omega$ is locally $C^{1,\alpha'(\cdot)}$, then $\psi^{-1}\in C^{1,\tilde\alpha(\cdot)}$ and the coefficients of $M$ are $C^{0,\gamma\circ\psi(\cdot)}(B_r^+)$. Therefore, if we obtain $C^{1,\gamma\circ\psi(\cdot)}_{\mathrm{loc}}(B_r^+\cup B'_r)$-regularity for $v$, then it translates into $C^{1,\gamma(\cdot)}_{\mathrm{loc}}(B_r\cap\overline\Omega)$-regularity for $u$. This is due to the fact that  
\begin{equation*}
\nabla u=\nabla v(\psi^{-1})\cdot J_{\psi^{-1}}
\end{equation*}
is the product of a $C^{0,\gamma(\cdot)}_{\mathrm{loc}}(B_r\cap\overline\Omega)$ and a $C^{0,\tilde\alpha(\cdot)}_{\mathrm{loc}}(B_r\cap\overline\Omega)$ functions. In other words, Theorem \ref{teo2} implies Corollary \ref{cor1} part $ii)$.
\end{itemize}
\end{proof}

\subsection{A family of regularized problems}\label{sectrego}
First of all, we would like to recall the standard definition of weak solution to problem \eqref{eq1}. In case of Dirichlet boundary conditions a function $u\in H^1(B_r^+)$ is a weak solution to \eqref{eq1} if for any test function $\phi\in C^\infty_c(B_r^+)$
\begin{equation*}
\int_{B_r^+}A\nabla u\cdot\nabla\phi=\int_{B_r^+}f\phi+\int_{B_r^+}F\cdot\nabla\phi+\int_{B_r^+}Vu\phi+\int_{B_r^+}b\cdot\nabla u\phi.
\end{equation*}
Additionally we ask that $\mathrm{Tr} u$, which is an element in $H^{1/2}(\partial B_r^+)$, coincides with $g$ on $B_r'$. In case of Neumann boundary conditions, a weak solution to \eqref{eq1} satisfies for any test function $\phi\in C^\infty_c(B_r^+\cup B'_r)$ 
\begin{equation*}
\int_{B_r^+}A\nabla u\cdot\nabla\phi=\int_{B_r^+}f\phi+\int_{B_r^+}F\cdot\nabla\phi+\int_{B_r^+}Vu\phi+\int_{B_r^+}b\cdot\nabla u\phi+\int_{B_r'}h\phi.
\end{equation*}

It is not our interest to write precisely what the minimal conditions on the data are in order to give sense to the above definition. We just would like to observe that the integrability and regularity assumptions that we will make on the data are amply sufficient to guarantee the validity of the weak formulation of the problem \eqref{eq1}.

In this section we are going to introduce a family of regularized problems related to \eqref{eq1}. This regularization is done by convolving variable coefficients and the given data which belongs to variable exponent H\"older spaces with a standard family of mollifiers and instead approximating the given data which belongs to variable exponent Lebesgue spaces using some density result of smooth functions in the space.

Every time we are considering some data in a variable exponent H\"older space of the half ball $B_r^+$ (or variable coefficients), imagine to extend this function across $\{x_n=0\}$ in an even way. This operation is merely technical, does not affect the given regularity and allows to define the convolution up to the flat boundary. In fact, this way we can ensure uniform convergence of mollifications of our data to the data themselves on the compact $\overline{B^+_{\overline r}}$ with $0<\overline r<r$. 
Let
\begin{equation}\label{regucoe}
A_\varepsilon(x)=(a_{ij}^\varepsilon(x))_{i,j=1}^{n},\qquad\mathrm{with} \quad a_{ij}^\varepsilon(x)=a_{ij}\ast\overline\eta_{\varepsilon}(x)=\int_{\R^{n}}\overline\eta(t)a_{ij}(x-\varepsilon t)\mathrm{d}t,
\end{equation}
where $\overline\eta\in C^{\infty}_c(\R^{n})$
is a nonnegative radially decreasing cut-off function with $\int_{\R^{n}}\overline\eta=1$ and the standard family of mollifiers is given by
\begin{equation*}
\overline\eta_\varepsilon(x)=\frac{1}{\varepsilon^{n}}\overline\eta\left(\frac{x}{\varepsilon}\right).
\end{equation*}
For example we can imagine that $\mathrm{supp}\overline\eta=B_1$. The mollifications defined above are well defined and smooth $A_\varepsilon\in C^\infty(\overline{B_{\overline r}})$ in any fixed ball with radius $0<\overline r< r$ provided that $0<\varepsilon\leq\overline\varepsilon$ for a certain $\overline\varepsilon$ which depends on $\overline r$ ($\overline\varepsilon\searrow0$ as $\overline r\nearrow r$). Moreover, if for example $A\in C(B_r)$ then $A_\varepsilon\to A$ uniformly in $\overline{B_{\overline r}}$.

We define in the same way $F_\varepsilon$ when $F\in C^{0,\alpha(\cdot)}(B_r^+)$. Nevertheless, for the Dirichlet boundary data $g$ and the Neumann boundary data $h$ (in case we are assuming $h\in C^{0,\alpha(\cdot)}(B'_r)$) we define $g_\varepsilon,h_\varepsilon$ by convolution with a standard family of mollifiers in $\R^{n-1}$; that is, for instance
\begin{equation}\label{regucoe1}
g_\varepsilon(x')=g\ast\tilde\eta_{\varepsilon}(x')=\int_{\R^{n-1}}\tilde\eta(t)g(x'-\varepsilon t)\mathrm{d}t 
\end{equation}
where $\tilde\eta\in C^{\infty}_c(\R^{n-1})$
is a nonnegative radially decreasing cut-off function with $\int_{\R^{n-1}}\tilde\eta=1$ and
\begin{equation*}
\tilde\eta_\varepsilon(x')=\frac{1}{\varepsilon^{n-1}}\tilde\eta\left(\frac{x'}{\varepsilon}\right).
\end{equation*}

Now we are going to present the key lemma for regularization in variable exponent H\"older spaces. The result below ensures a uniform boundedness in variable exponent H\"older spaces for the family of mollifications, even if such spaces are not translation invariant. The main idea is the following: there are two scales of infinitesimals which are independent; they are the distance between points $|x-y|$ and the parameter of mollification $\varepsilon$. When the relationship between the two scales is $|x-y|\gtrsim\varepsilon$, then the translated point $x-\varepsilon t$ is not "too far" from $x$ and the $\log$-H\"older continuity of $\alpha$ is enough to control the difference $(\alpha(x-\varepsilon t)-\alpha(x))\log|x-y|$. Otherwise, when $|x-y|\lesssim\varepsilon$, the translation above can not be controlled without the help of the Lipschitz continuity of the cut-off function.

\begin{Lemma}\label{keylemma}
Let $0<\overline r<r\leq 1/2$, $k\in\mathbb N$, $0<\varepsilon\leq\overline\varepsilon(\overline r)<1$ and $\alpha\in\mathcal{A}^{\log}(B_r)$. If $u\in C^{k,\alpha(\cdot)}(B_r)$, then mollifications $u_\varepsilon=u\ast\overline\eta_\varepsilon$ defined in \eqref{regucoe} satisfy the following
$$\|u_\varepsilon\|_{C^{k,\alpha(\cdot)}(B_{\overline r})}\leq c, $$
for a constant $c>0$ which does not depend on $\varepsilon\leq\overline\varepsilon$.
\end{Lemma}
\proof
It is enough to prove the statement in case $k=0$, then the same reasoning applies also to partial derivatives of any order. The uniform $L^\infty$ bound is trivial, in fact for $x\in B_{\overline r}$
\begin{equation*}
|u_\varepsilon(x)|\leq \int_{\R^n}\overline\eta(t)|u(x-\varepsilon t)|\mathrm{d}t\leq \|u\|_{L^\infty(B_r)}.
\end{equation*}
Now we want to estimate uniformly in $\varepsilon$ the $\alpha(\cdot)$-H\"older seminorm of $u_\varepsilon$ in $B_{\overline r}$. Hence, for any triplet $(x,y,\varepsilon)\in B_{\overline r}\times B_{\overline r}\times (0,\overline\varepsilon]=:A$, we have that it belongs to one of the following two subsets
$$A_1=\{(x,y,\varepsilon)\in A \ : \ |\log|x-y||\geq \frac{n+1}{1-\overline\alpha}|\log\varepsilon|\}$$
or
$$A_2=\{(x,y,\varepsilon)\in A \ : \ |\log|x-y||< \frac{n+1}{1-\overline\alpha}|\log\varepsilon|\}.$$
If we have a triplet in $A_1$, we use the Lipschitz continuity of $\overline\eta$ in the estimate below
\begin{eqnarray*}
|u_\varepsilon(x)-u_\varepsilon(y)|&\leq&\int_{\R^n}|\overline\eta_\varepsilon(x-t)-\overline\eta_\varepsilon(y-t)|\cdot|u(t)|\mathrm{d}t\\
&\leq&\frac{c|x-y|}{\varepsilon^{n+1}}\\
&=&c|x-y|^{\alpha(x)}|x-y|^{\overline\alpha-\alpha(x)}\frac{|x-y|^{1-\overline\alpha}}{\varepsilon^{n+1}}\\
&\leq&c|x-y|^{\alpha(x)}\leq c|x-y|^{\max\{\alpha(x),\alpha(y)\}},
\end{eqnarray*}
where in the previous estimates we used the condition which characterizes $A_1$ in order to bound uniformly $\frac{|x-y|^{1-\overline\alpha}}{\varepsilon^{n+1}}\leq 1$, and in the last inequality we used the $\log$-H\"older continuity of exponent $\alpha$. In fact, if $\alpha(x)=\max\{\alpha(x),\alpha(y)\}$ we have done, otherwise we can bound uniformly
$$(\alpha(x)-\alpha(y))\log|x-y|=|\alpha(x)-\alpha(y)|\cdot|\log|x-y||\leq \omega(|x-y|)|\log|x-y||\leq c.$$
If we have a triplet in $A_2$, we notice that
$$\omega(\varepsilon)|\log|x-y||< \frac{n+1}{1-\overline\alpha}\omega(\varepsilon)|\log\varepsilon|\leq c.$$
Hence, we show that there exists a positive constant independent from $t\in \mathrm{supp}\overline \eta= B_1$ and from the triplet such that
\begin{equation*}
|x-y|^{\alpha(x-\varepsilon t)}\leq c |x-y|^{\alpha(x)}.
\end{equation*}
If $\alpha(x-\varepsilon t)>\alpha(x)$ the inequality is trivial. Otherwise, it is enough to estimate uniformly
$$(\alpha(x-\varepsilon t)-\alpha(x))\log|x-y|=|\alpha(x-\varepsilon t)-\alpha(x)|\cdot|\log|x-y||\leq\omega(\varepsilon|t|)\cdot|\log|x-y||\leq\omega(\varepsilon)\cdot|\log|x-y||\leq c.$$
Hence, we can conclude that
\begin{eqnarray*}
|u_\varepsilon(x)-u_\varepsilon(y)|&\leq&\int_{\R^n}\overline\eta(t)|u(x-\varepsilon t)-u(y-\varepsilon t)|\mathrm{d}t\\
&\leq&c\int_{\R^n}\overline\eta(t)|x-y|^{\max\{\alpha(x-\varepsilon t),\alpha(y-\varepsilon t)\}}\mathrm{d}t\\
&=&c\int_{\R^n}\overline\eta(t)|x-y|^{\alpha(x-\varepsilon t)}\mathrm{d}t\\
&\leq&c|x-y|^{\alpha(x)}\leq c|x-y|^{\max\{\alpha(x),\alpha(y)\}}.
\end{eqnarray*}
Above we have assumed without loss of generality that $\alpha(x-\varepsilon t)=\max\{\alpha(x-\varepsilon t),\alpha(y-\varepsilon t)\}$, otherwise one can change the roles of $x$ and $y$.
\endproof

The idea now is to regularize the data which belong to variable exponent Lebesgue spaces by \cite[Theorem 3.4.12]{Die}; that is, density of $C^{\infty}_c(\Omega)$ in $L^{p(\cdot)}(\Omega)$ where $\Omega$ is open. Hence, when forcing terms, fields, potentials and drifts $f,F,V,b$ belong to certain variable exponent Lebesgue spaces $L^{p(\cdot)}(B_r^+)$, then we define sequences of functions $f_k,F_k,V_k,b_k$ in $C^{\infty}_c(B_r^+)$ strongly converging in the Lebesgue spaces. When $h\in L^{s(\cdot)}(B'_r)$, then we define a sequence of functions $h_k$ in $C^{\infty}_c(B'_r)$ strongly converging in the Lebesgue space.

\subsubsection{The regularization for H\"older estimates}
We are dealing here with given data which satisfy integrability and regularity conditions in Theorem \ref{teo1}. Hence, as $k\to+\infty$ we consider a sequence $\varepsilon_k\to0^+$. Along such a sequence we study problems
\begin{equation}\label{poisseps}
\begin{cases}
-\mathrm{div}\left(A_{\varepsilon_{k}}\nabla u_k\right)=f_k+\mathrm{div}F_{k}+V_ku_k+b_k\cdot\nabla u_k &\mathrm{in} \ B_r^+\\
u_k=g_{\varepsilon_{k}} \quad\mathrm{or}\quad A_{\varepsilon_{k}}\nabla u_k\cdot\nu=h_k &\mathrm{on} \ B'_r,
\end{cases}
\end{equation}
where $f_k,F_k,V_k,b_k,h_k$ are sequences of smooth functions which strongly converges in the right variable exponent Lebesgue spaces to $f,F,V,b,h$, while $A_{\varepsilon_k},g_{\varepsilon_k}$ are mollifications of $A,g$ as in \eqref{regucoe} and \eqref{regucoe1}.
\begin{remark}
We would like to notice here that weak solutions to \eqref{poisseps} are uniformly-in-$k$ locally bounded by Moser iterations assuming the integrability and regularity conditions on data given in Theorem \ref{teo1}. This can be done by Sobolev and trace inequalities, using the fact that the quadratic forms related to coefficients $A_{\varepsilon_k}$ are equivalent norms in the Sobolev space uniformly-in-$k$. Nonetheless, we remark here that the space $L^{p(\cdot)}$ trivially embeds into $L^{\underline{p}}$ and $C^{0,\alpha(\cdot)}$ into $C^{0,\underline\alpha}$. In other words, it is very easy to check that there exists a uniform-in-$k$ constant such that
\begin{equation}\label{Linfeps}
\|u_k\|_{L^\infty(B^+_{r/2})}\leq c\left(\|u_k\|_{L^2(B^+_r)}+\|f_k\|_{L^{p(\cdot)}(B^+_r)}+\|F_k\|_{L^{q(\cdot)}(B^+_r)}+\begin{cases}\|g_{\varepsilon_k}\|_{C^{0,\alpha(\cdot)}(B'_r)}\quad\mathrm{or}\\
\|h_k\|_{L^{s(\cdot)}(B'_r)}\end{cases}\right).
\end{equation}
For this reason, without loss of generality we will always assume a uniform-in-$k$ $L^\infty$-bound for weak solutions to \eqref{poisseps}. We notice that the constant in \eqref{Linfeps} depends on the norms of the potential and transport terms and so is uniform if we assume uniform bounds $\|V_k\|_{L^{m_1(\cdot)}(B_r^+)}\leq c_1$ and $\|b_k\|_{L^{m_2(\cdot)}(B_r^+)}\leq c_2$. For the requirement $\underline m_2>n$ we refer to \cite{Tru}. For the sharp requirement $\underline s>n-1$ we refer for example to \cite[Proposition 2.6. (i)]{JinLiXio}.
\end{remark}

Hence, solutions to \eqref{poisseps} satisfy the suboptimal estimate
\begin{equation}\label{C0alphaesteps}
\|u_k\|_{C^{0,\alpha(\cdot)}(B^+_{r/2})}\leq c\left(\|u_k\|_{L^\infty(B^+_r)}+\|f_k\|_{L^{p(\cdot)}(B^+_r)}+\|F_k\|_{L^{q(\cdot)}(B^+_r)}+\begin{cases}\|g_{\varepsilon_k}\|_{C^{0,\alpha(\cdot)}(B'_r)}\quad\mathrm{or}\\
\|h_k\|_{L^{s(\cdot)}(B'_r)}\end{cases}\right),
\end{equation}
with a constant which depends on $\|V_k\|_{L^{m_1(\cdot)}(B_r^+)}$ and $\|b_k\|_{L^{m_2(\cdot)}(B_r^+)}$ and possibly is not uniform in $k$. This kind of estimate, fixed $k>0$, is implied by classical boundary regularity results for second order uniformly elliptic equations in divergence form with smooth coefficients and data (this is usually done by "freezing the coefficients" and passing to constant ones, see for instance \cite[Chapter 6]{GilTru}, \cite{Amb} or \cite{GiaMar}). In the following result, we are going to prove that actually the constant in the estimate above \eqref{C0alphaesteps} can be taken uniform as $k\to+\infty$. The proof is based on a contradiction argument involving a blow-up procedure and a Liouville type theorem for harmonic functions (this kind of argument is very well known and we refer to \cite[Chapter 2]{FerRos} for a nice and clear overview of the technique).

\begin{Proposition}\label{prop1}
As $k\to+\infty$ let $\{u_k\}$ be a family of solutions to \eqref{poisseps} in $B^+_r$ with $r\leq1/2$. Let $p,q,m_1,m_2\in\mathcal{P}^{\log}(B^+_r)$ with $\underline p,\underline m_1>\frac{n}{2}$, $\underline q,\underline m_2>n$. Let $s\in\mathcal{P}^{\log}(B'_r)$ with $\underline s>n-1$. Let $\|V_k\|_{L^{m_1(\cdot)}(B_r^+)}\leq c_1$ and $\|b_k\|_{L^{m_2(\cdot)}(B_r^+)}\leq c_2$. Let also $\alpha\in \mathcal{A}^{\log}(B^+_r)$ with \eqref{alpha1}. Then, there exists a positive constant
independent from $k$ such that \eqref{C0alphaesteps} holds true.
\end{Proposition}
\proof
Without loss of generality we can assume the existence of a uniform constant $c>0$ as $k\to+\infty$ such that the terms in the right hand side in \eqref{C0alphaesteps} are uniformly bounded; that is,
\begin{equation*}
\|u_k\|_{L^\infty(B^+_r)}+\|f_k\|_{L^{p(\cdot)}(B^+_r)}+\|F_k\|_{L^{q(\cdot)}(B^+_r)}+\begin{cases}\|g_{\varepsilon_k}\|_{C^{0,\alpha(\cdot)}(B'_r)}\quad\mathrm{or}\\
\|h_k\|_{L^{s(\cdot)}(B'_r)}\end{cases}\leq c.
\end{equation*}
If this is not the case, the thesis is trivially satisfied. We have already assumed that 
\begin{equation*}
\|V_k\|_{L^{m_1(\cdot)}(B^+_r)}+\|b_k\|_{L^{m_2(\cdot)}(B^+_r)}\leq c.
\end{equation*}
{\bf Step 1: the contradiction argument.} We argue by contradiction; that is, there exists $\alpha\in \mathcal{A}^{\log}(B^+_r)$ with \eqref{alpha1}, and a subsequence of solutions (always denoted by $\{u_k\}$) to \eqref{poisseps} such that
$$\|\eta u_k\|_{C^{0,\alpha(\cdot)}(B^+_r)}\to+\infty,$$
where the function $\eta\in C^\infty_c(B_r)$ is a radial decreasing cut-off function such that $\eta\equiv1$ in $B_{r/2}$, $0\leq\eta\leq1$ in $B_r$ and $\mathrm{supp}\eta:=B=B_{3r/4}$. Moreover we can take $\eta\in C^{0,1}(B)$ such that $\eta(z)\leq \ell\mathrm{dist}(z,\partial B)$ where $\ell$ is the Lipschitz constant. Hence, we are supposing that
\begin{equation*}
\max_{\substack{z,\zeta \in B^+_{r}\\z\neq\zeta}}\frac{\abs{\eta u_k(z)-\eta u_k(\zeta)}}{\abs{z-\zeta}^{\alpha(z)}}=L_k \to +\infty.
\end{equation*}
We can assume that $L_k$ is attained by a couple of points $z_k,\zeta_k\in B\cap\{x_n\geq0\}$ and we call $r_k:=|z_k-\zeta_k|$. Hence
\begin{equation*}
\frac{\abs{\eta u_k(z_k)-\eta u_k(\zeta_k)}}{\abs{z_k-\zeta_k}^{\alpha(z_k)}}=L_k.
\end{equation*}
Using the $L^\infty$ bound of the $u_k$'s and the Lipschitz continuity of $\eta$, one can easily show that 
\begin{itemize}
\item[$i)$] $r_k\to0$,
\item[$ii)$] $\ddfrac{\mathrm{dist}(z_k,\partial^+ B^+)}{r_k}\to+\infty$   \  and   \   $\ddfrac{\mathrm{dist}(\zeta_k,\partial^+ B^+)}{r_k}\to+\infty$,
\end{itemize}
where $B^+=B_{3r/4}^+$ and $\partial^+B^+=\partial B_{3r/4}\cap\{x_n>0\}$. In fact, since $r_k=|z_k-\zeta_k|\leq 2r\leq1$, point $i)$ is implied by
\begin{equation*}
L_k=\frac{\abs{\eta u_k(z_k)-\eta u_k(\zeta_k)}}{r_k^{\alpha(z_k)}}\leq\frac{c}{r_k^{\overline\alpha}}.
\end{equation*}
For $ii)$ one can reason in the same way using also Lipschitz continuity of $\eta$.

%

\noindent {\bf Step 2: the blow-up sequences.} Let us define
\begin{equation*}
v_k(z)=\frac{\eta u_k(\hat z_k+r_kz)-\eta u_k(\hat z_k)}{L_kr_k^{\alpha(z_k)}},\qquad w_k(z)=\frac{\eta(\hat z_k)(u_k(\hat z_k+r_kz)- u_k(\hat z_k))}{L_kr_k^{\alpha(z_k)}},
\end{equation*}
with
$$z\in B(k):=\frac{B^+- \hat z_k}{r_k},$$
and $\hat z_k\in B\cap\{x_n\geq0\}$ to be determined.
We will take in any case $\hat z_k=(z'_k,\hat z_{k,n})$, where $z_k=(z'_k,z_{k,n})$. 
At this point, since $z_k\in B\cap\{x_n\geq0\}$, then the following quantity has a sign
\begin{equation}\label{rat1}
\frac{z_{k,n}}{r_k}\geq0.
\end{equation}
There are two cases:
\begin{itemize}
\item[{\bf Case 1 :}] the term in \eqref{rat1} is unbounded, then we choose $\hat z_{k,n}=z_{k,n}$. In other words, $\hat z_k=z_k$;
\item[{\bf Case 2 :}] the term in \eqref{rat1} is bounded, then we choose $\hat z_{k,n}=0$. In other words, $\hat z_k=(z'_k,0)$;
\end{itemize}

Now we want to understand the limit set $B(\infty)=\lim_{k\to +\infty} B(k)$. This limit of sets must be understood in the following way: $z\in B(\infty)$ if there exists $k_0$ such that $z\in B(k)$ for any $k>k_0$. The blow-up domain $B(k)$ can be expressed as the following intersection
\begin{equation*}
B(k)=\frac{B-\hat z_k}{r_k}\cap\frac{\{x_n>0\}-\hat z_k}{r_k}=B_{\frac{3}{4r_k}}\left(-\frac{\hat z_k}{r_k}\right)\cap H_k,
\end{equation*}
with
$$H_k=\left\{x_n>-\frac{\hat z_{k,n}}{r_k}\right\}.$$
First, we observe that
$$B_{\frac{3}{4r_k}}\left(-\frac{\hat z_k}{r_k}\right)\longrightarrow\R^{n}.$$
This is due to the fact that
\begin{equation*}\label{ballin0}
B_{\frac{\mathrm{dist}(\hat z_k,\partial^+ B^+)}{2}}\subset B-\hat z_k
\end{equation*}
and hence, using $ii)$, one has the result. In fact, in {\bf Case 1} this is immediate by the choice $\hat z_k=z_k$. Instead, in {\bf Case 2}
\begin{equation*}
+\infty\leftarrow\frac{\mathrm{dist}(z_k,\partial^+ B^+)}{r_k}\leq \frac{\mathrm{dist}(\hat z_k,\partial^+ B^+)}{r_k}+\frac{z_{k,n}}{r_k}\leq \frac{\mathrm{dist}(\hat z_k,\partial^+ B^+)}{r_k} + c.
\end{equation*}
In {\bf Case 1}
$$H_k\longrightarrow\R^{n}.$$
In {\bf Case 2}
$$H_k=\left\{x_n>0\right\}.$$

We remark here that since $\hat z_k\in B\cap\{x_n\geq0\}$, then $0\in \overline{B(k)}$ for any $k$. Let us consider $z,\zeta\in K\subset \overline{B(\infty)}$ in a compact set. Then
\begin{equation*}\label{hlungo}
|v_k(z)-v_k(\zeta)|=\frac{|\eta u(\hat z_k+r_kz)-\eta u(\hat z_k+r_k\zeta)|}{L_kr_k^{\alpha(z_k)}}\leq |z-\zeta|^{\alpha(\hat z_k+r_kz)}r_k^{\alpha(\hat z_k+r_kz)-\alpha(\hat z_k)}r_k^{\alpha(\hat z_k)-\alpha(z_k)}.
\end{equation*}
For any $k$, we have named $z$ as the point such that $\alpha(\hat z_k+r_kz)\geq\alpha(\hat z_k+r_k\zeta)$ (the roles of $z$ and $\zeta$ are interchangeable here). Using the $\log$-H\"older continuity of exponent $\alpha$, we can bound uniformly in the compact set
\begin{equation*}
r_k^{\alpha(\hat z_k+r_kz)-\alpha(\hat z_k)}\leq c(K).
\end{equation*}
In fact
\begin{equation*}
|\alpha(\hat z_k+r_kz)-\alpha(\hat z_k)|\cdot|\log r_k|\leq\omega(r_k|z|)|\log r_k|\leq\omega\left(r_k\max_K|z|\right)|\log r_k|\leq c.
\end{equation*}
Nevertheless, in {\bf Case 2}, when $\hat z_k\neq z_k$ and $z_{k,n}/r_k$ is bounded, we can bound also
\begin{equation}\label{cappucciononcappuccio}
r_k^{\alpha(\hat z_k)-\alpha(z_k)}\leq c.
\end{equation}
In fact
\begin{equation*}
|\alpha(\hat z_k)-\alpha(z_k)|\cdot|\log r_k|\leq\omega(z_{k,n})|\log r_k|=\omega\left(r_k\frac{z_{k,n}}{r_k}\right)|\log r_k|\leq c.
\end{equation*}
Hence there exists $k_0$ such that for $k>k_0$
\begin{equation}\label{condition1}
\max_{\substack{z,\zeta \in K\\ z\neq\zeta}}\frac{\abs{v_k(z)-v_k(\zeta)}}{\abs{z-\zeta}^{\alpha(\hat z_k+r_kz)}}\leq c(K).
\end{equation}
This condition gives uniform boundedness and uniform equicontinuity of the sequence $v_k$ on any compact set $K$ containing the origin. In fact, $v_k(0)=0$ for any $k$ and
\begin{equation*}
|v_k(z)|\leq c(K)|z|^{\max\{\alpha(\hat z_k+r_kz),\alpha(\hat z_k)\}}\leq c(K)\max\{1,|z|^{\overline\alpha}\}\leq\tilde c(K).
\end{equation*}
Moreover, fixed $\varepsilon>0$, taken $\delta=\min\{1,(\varepsilon/c(K))^{1/\underline\alpha}\}$ and $z,\zeta\in K$ with $|z-\zeta|<\delta$
\begin{equation*}
|v_k(z)-v_k(\zeta)|\leq c(K)|z-\zeta|^{\alpha(\hat z_k+r_kz)}<\varepsilon.
\end{equation*}
Then, by the Ascoli-Arzel\'a theorem, there exists a subsequence converging uniformly on the compact $K$ to a function $v$. By a compact exhaustion of $\overline{B(\infty)}$ with $K_1\subset K_2\subset K_3...$ and a diagonal procedure, we can extract a subsequence converging on any compact set to the entire profile $v$.

Moreover, 
 
\begin{equation}\label{alpha10}
\left|v_k\left(\frac{z_k- \hat z_k}{r_k}\right)-v_k\left(\frac{\zeta_k- \hat z_k}{r_k}\right)\right|=1\;
\end{equation}
for any $k$. Hence, we have that the limit $v$ is a non constant. In fact, up to pass to a subsequence, in {\bf Case 1} $\frac{\zeta_k-z_k}{r_k}\to\overline z\in \mathbb S^{n-1}$ since any point of the sequence belongs to $\mathbb S^{n-1}$. Hence, by \eqref{alpha10}, uniform convergence and equicontinuity, we have $|v(0)-v(\overline z)|=1$. In {\bf Case 2}, up to pass to a subsequence, the sequence 
$$\frac{z_k-\hat z_k}{r_k}=\frac{(0,z_{k,n})}{r_k}\to(0,0,c)=z_1,$$
and
$$\frac{\zeta_k-\hat z_k}{r_k}=\frac{\zeta_k-z_k}{r_k}+\frac{(0,z_{k,n})}{r_k}\to \overline z+(0,0,c)=z_2,$$
where $\overline z\in \mathbb S^{n-1}$. Hence, by \eqref{alpha10}, uniform convergence and equicontinuity, again $|v(z_1)-v(z_2)|=1$.

%
%
Let us define $z_\infty=\lim z_k=\lim \hat z_k$, up to pass to subsequences, and $\alpha_\infty=\alpha(z_\infty)$. Taken $z,\zeta\in \overline{B(\infty)}$, they are contained in a compact set containing the origin $K$. Hence, passing to the pointwise limit in inequality \eqref{condition1}, we have
\begin{equation*}
\frac{\abs{v(z)-v(\zeta)}}{\abs{z-\zeta}^{\alpha_\infty}}\leq e^2.
\end{equation*}
This implies that $v$ is globally $\alpha_\infty$-H\"older continuous in $B(\infty)$. Nevertheless, $w_k\to v$ do converge to the same limit on compact sets since
\begin{equation*}
\sup_{z\in K}|v_k(z)-w_k(z)|\to0.
\end{equation*}
{\bf Step 3: the rescaled equations.} Let $k>k_0$ large enough. The functions $w_k$ solve in $B(k)$
\begin{eqnarray}\label{Lakw0}
-\mbox{div}\!\left(A_{\varepsilon_k}(\hat z_k+r_k\cdot)\nabla w_k\right)(z)&=&\frac{\eta(\hat z_k)}{L_k}r_k^{2-\alpha(z_k)}f_{k}(\hat z_k+r_kz)+\frac{\eta(\hat z_k)}{L_k}r_k^{1-\alpha(z_k)}\mbox{div}\!\left(F_k(\hat z_k+r_k\cdot)\right)(z)\nonumber\\
&&+\frac{\eta(\hat z_k)}{L_k}r_k^{2-\alpha(z_k)}V_{k}u_k(\hat z_k+r_kz)+r_kb_k(\hat z_k+r_kz)\cdot\nabla w_k(z)
\end{eqnarray}
in {\bf Case 1}, while in {\bf Case 2} it remains also the boundary condition at $\{x_n=0\}$.
We show that the right hand side in the rescaled equation is vanishing in $L^1_{\mathrm{loc}}$ (let us work for simplicity in {\bf Case 1}) in the sense that taken a test function $\phi\in C^\infty_c(\R^n)$, with $\mathrm{supp}\phi\subset B_R$ for a certain $R>0$
\begin{eqnarray}\label{L1loc}
\frac{\eta(z_k)}{L_k}r_k^{2-\alpha(z_k)}\left|\int_{B_R}f_{k}(z_k+r_kz)\phi(z)\mathrm{d}z\right|&\leq& cr_k^{2-\alpha(z_k)-n}\int_{B_{r_kR}(z_k)}|f_{k}(\zeta)|\mathrm{d}\zeta\nonumber\\
&\leq& \frac{c}{L_k}r_k^{2-\alpha(z_k)-n}\|f_{k}\|_{L^{p(\cdot)}(B_{r_kR}(z_k))}\|\chi_{B_{r_kR}(z_k)}\|_{L^{p'(\cdot)}(\R^n)}\nonumber\\
&\leq& \frac{c}{L_k}r_k^{2-\alpha(z_k)-n}|B_{r_kR}(z_k)|^{\avint_{B_{r_kR(z_k)}}\frac{1}{p'}}\nonumber\\
&\leq& \frac{c}{L_k}r_k^{2-\alpha(z_k)-\frac{n}{p(z_k)}}r_k^{n\left(\frac{1}{p(z_k)}-\avint_{B_{r_kR(z_k)}}\frac{1}{p}\right)}\leq \frac{c}{L_k}.
\end{eqnarray}
In fact, by $\log$-H\"older continuity of $1/p$ we have
\begin{equation*}
\left|\frac{1}{p(z_k)}-\avint_{B_{r_kR(z_k)}}\frac{1}{p(z)}\mathrm{d}z\right|\leq \avint_{B_{r_kR(z_k)}}\frac{c}{\log(e+\frac{1}{|z_k-z|})}\mathrm{d}z\leq\frac{c}{\log(e+\frac{1}{|z_k-(z_k+r_kR\xi)|})}
\end{equation*}
for $\xi\in \mathbb S^{n-1}$. For the estimate above on the norm of characteristic functions, in term of Lebesgue measure of the set elevated to the average of the reciprocal of the exponent, we refer to \cite[Section 4.5]{Die}. The term in the right hand side with the potential vanishes using the same argument, once we recall that the $u_k$'s are uniformly bounded. With very similar reasonings we can deal with the field term
\begin{eqnarray*}
\frac{\eta(z_k)}{L_k}r_k^{1-\alpha(z_k)}\left|\int_{B_R}F_{k}(z_k+r_kz)\cdot\nabla\phi(z)\mathrm{d}z\right|&\leq& cr_k^{1-\alpha(z_k)-n}\int_{B_{r_kR}(z_k)}|F_{k}(\zeta)|\mathrm{d}\zeta\nonumber\\
&\leq& \frac{c}{L_k}r_k^{1-\alpha(z_k)-n}\|F_{k}\|_{L^{q(\cdot)}(B_{r_kR}(z_k))}\|\chi_{B_{r_kR}(z_k)}\|_{L^{q'(\cdot)}(\R^n)}\nonumber\\
&\leq& \frac{c}{L_k}r_k^{1-\alpha(z_k)-n}|B_{r_kR}(z_k)|^{\avint_{B_{r_kR(z_k)}}\frac{1}{q'}}\nonumber\\
&\leq& \frac{c}{L_k}r_k^{1-\alpha(z_k)-\frac{n}{q(z_k)}}r_k^{n\left(\frac{1}{q(z_k)}-\avint_{B_{r_kR(z_k)}}\frac{1}{q}\right)}\leq \frac{c}{L_k}.
\end{eqnarray*}
Now we show that, in order to make vanish the drift term, no restriction on $\alpha$ is needed.
\begin{eqnarray*}
&&r_k\left|\int_{B_R}b_{k}(z_k+r_kz)\cdot\nabla w_k(z)\phi(z)\mathrm{d}z\right|\nonumber\\
&\leq& cr_k^{1-n/2}\left(\int_{B_R}|\nabla w_k|^2\right)^{1/2}\|b_{k}\|_{L^{m_2(\cdot)}(B_{r_kR}(z_k))}\|\chi_{B_{r_kR}(z_k)}\|_{L^{2m_2(\cdot)/(m_2(\cdot)-2)}(\R^n)}\nonumber\\
&\leq&cr_k^{1-n/2}|B_{r_kR}(z_k)|^{\avint_{B_{r_kR(z_k)}}\frac{m_2-2}{2m_2}}\left(\int_{B_R}|\nabla w_k|^2\right)^{1/2}\nonumber\\
&\leq&cr_k^{1-\frac{n}{m_2(z_k)}}\left(\int_{B_R}|\nabla w_k|^2\right)^{1/2}=t_k\left(\int_{B_R}|\nabla w_k|^2\right)^{1/2},
\end{eqnarray*}
with $t_k\to0$. We will show that the full term vanishes thanks to a uniform energy bound of the $w_k$'s in \eqref{cacioppoli0}.

We remark that in {\bf Case 2} nothing changes apart from notations, up to suitably adjusting terms using \eqref{cappucciononcappuccio}.

\noindent {\bf Step 4: the limit equation.}
We define the limit constant coefficients matrix
\begin{equation*}
\hat A=A(z_\infty)=\lim_{k\to+\infty} A_{\varepsilon_k}(\hat z_k)=\lim_{k\to+\infty} A_{\varepsilon_k}(\hat z_k+r_kz).
\end{equation*}

By a Caccioppoli type inequality, easily obtained by multiplying \eqref{Lakw0} by $\varphi w_k$, being $\varphi$ a cut-off function, taking into account possible boundary conditions, that functions $w_k$ are uniformly bounded and the vanishing of right hand sides (the drift term can be reabsorbed having $t_k\to0$), then we obtain uniform-in-$k$ energy bounds holding on compact subsets of $B(\infty)$
\begin{equation}\label{cacioppoli0}
\forall R>0,\;\exists c>0:\;\forall k, \qquad
\int_{B_R\cap B(\infty)}A_{\varepsilon_k}(\hat z_k+r_k z) \nabla w_k\cdot \nabla w_k\leq c\;.
\end{equation}
In {\bf Case 1}, $B(\infty)=\R^{n}$. Then the limit $v$ belongs to $H^1_{\mathrm{loc}}(\R^{n})$ and is a weak solution of
\begin{equation*}\label{entire0}
-\mathrm{div}\left(\hat A\nabla v\right)=0\qquad \mathrm{in} \ \R^{n},
\end{equation*}
in the sense that for every $\phi\in C^\infty_c(\R^{n})$
$$\int_{\R^{n}}\hat A\nabla v\cdot\nabla\phi=0.$$
In {\bf Case 2} we have $B(\infty)=\{x_n>0\}=\R^n_+$. Then the limit $v$ belongs to $(H^1_{\{x_n=0\}})_{\mathrm{loc}}(\overline{\R^{n}_+})$ or $H^1_{\mathrm{loc}}(\overline{\R^{n}_+})$ and is a weak solution of
\begin{equation*}\label{entire01}
\begin{cases}
-\mathrm{div}\left(\hat A\nabla v\right)=0 & \mathrm{in} \ \R^{n}_+,\\
v=0 \qquad\mathrm{or}\qquad \hat A\nabla v\cdot\nu=0 &\mathrm{on} \ \{x_n=0\},
\end{cases}
\end{equation*}
in the sense that for every $\phi\in C^\infty_c(\R^{n}_+)$ or $\phi\in C^\infty_c(\overline{\R^{n}_+})$
$$\int_{\R^{n}_+}\hat A\nabla v\cdot\nabla\phi=0.$$

As we have already said, this is done using the Caccioppoli inequality \eqref{cacioppoli0}, which gives weak convergence in $H^1_{\mathrm{loc}}$. Moreover, we use the pointwise convergence of coefficients $A_{\varepsilon_k}(\hat z_k+r_k\cdot)\to \hat A$, hence 
\begin{equation*}\label{conv10}
\int_{B(\infty)}A_{\varepsilon_k}(\hat z_k+r_k\cdot)\nabla w_k\cdot\nabla\phi\to\int_{B(\infty)}\hat A\nabla v\cdot\nabla\phi.
\end{equation*}
When we consider Dirichlet boundary conditions in {\bf Case 2}, thanks to Lemma \ref{keylemma} we have on compact sets $K'$ of the hyperplane $\{x_n=0\}$ a uniform constant $c(K')$ such that
\begin{eqnarray*}
|w_k(x',0)|&=&\frac{\eta(\hat z_k)}{L_kr_k^{\alpha(z_k)}}\left|g_{\varepsilon_k}(\hat z_k+r_k(x',0))-g_{\varepsilon_k}(\hat z_k))\right|\\
&\leq&c(K')\frac{\eta(\hat z_k)}{L_kr_k^{\alpha(z_k)}}|r_k(x',0)|^{\max\{\alpha(\hat z_k+r_k(x',0)),\alpha(\hat z_k)\}}\to0
\end{eqnarray*}
by the $\log$-H\"older continuity of $\alpha$ and using \eqref{cappucciononcappuccio}.
In case of Neumann boundary conditions, using integration by parts in the equation, the boundary term vanishes having $\underline s>n-1$ and $\alpha(x',0)\leq 1-\frac{n-1}{s(x',0)}$; that is,
\begin{eqnarray*}
\int_{B_R\cap\partial B(\infty)}A_{\varepsilon_k}(\hat z_k+r_kz)\nabla w_k(z)\cdot\nu\phi(z)&=&\frac{\eta(\hat z_k)r_k^{1-\alpha(\hat z_k)}r_k^{\alpha(\hat z_k)-\alpha(z_k)}}{L_k}\int_{B_R\cap\partial B(\infty)}h_{\varepsilon_k}(\hat z_k+r_kz)\phi(z)\\
&\leq&c \frac{r_k^{1-\alpha(\hat z_k)-\frac{n-1}{s(\hat z_k)}}}{L_k}\to0.
\end{eqnarray*}

\noindent {\bf Step 5: a square root reflection and a Liouville type theorem.}
The limit matrix $\hat A$ with constant coefficients is symmetric positive definite. Let us consider the square root $C=\sqrt{\hat A}$ of $\hat A$, which is symmetric and positive definite too. We remark that the linear transform associated to the inverse of such a matrix, maps $\R^{n}$ in itself, and in case $B(\infty)=\{x_n>0\}$, maps such half space in another half space
$$C^{-1}B(\infty)=C^{-1}\{x_n>0\}=\{\zeta \ : \ (C\zeta)_{n}>0\},$$
and the hyperplane $\{x_n=0\}$ in the boundary hyperplane $C^{-1}\{x_n=0\}=\{\zeta \ : \ (C\zeta)_{n}=0\}$. In other words, let $u(\zeta)=v(C \zeta)$, where $C \zeta=z$. The function $u$ is a weak solution of
\begin{equation*}\label{entire20}
-\Delta u(\zeta)=0\qquad \mathrm{in} \ C^{-1}B(\infty),
\end{equation*}
is non constant and globally $\alpha_\infty$-H\"older continuous. Additionally, in {\bf Case 2} it vanishes or has vanishing normal derivative on $C^{-1}\{x_n=0\}$; that is, weakly solves
\begin{equation*}\label{entire200}
\begin{cases}
-\Delta u(\zeta)=0 & \mathrm{in} \ C^{-1}\{x_n>0\}\\
u=0\qquad\mathrm{or}\qquad\nabla u\cdot\nu'=0 & \mathrm{in} \ C^{-1}\{x_n=0\},
\end{cases}
\end{equation*}
where $\nu'(\zeta)=C\nu(C\cdot \zeta)$ is the outward normal vector to the boundary hyperplane. Hence, in the last situation we can perform an odd or even reflection across that hyperplane, having in any case the equation satisfied in the whole of $\R^{n}$. Hence we have proved that the limit $u\in H^1_{\mathrm{loc}}(\mathbb{R}^{n})$ is not constant and globally harmonic in $\mathbb{R}^{n}$. Moreover it is globally $\alpha_\infty$-H\"older continuous with $\alpha_\infty<1$, in clear contradiction with the Liouville theorem in \cite[Corollary 2.3]{NorTavTerVer}.
\endproof

\subsubsection{The regularization for gradient estimates}

We are dealing here with given data which satisfy integrability and regularity conditions in Theorem \ref{teo2}. Hence, as $k\to+\infty$ we consider a sequence $\varepsilon_k\to0^+$. Along such a sequence we study problems
\begin{equation}\label{poiss2eps}
\begin{cases}
-\mathrm{div}\left(A_{\varepsilon_{k}}\nabla u_k\right)=f_k+\mathrm{div}F_{\varepsilon_k}+V_ku_k+b_k\cdot\nabla u_k &\mathrm{in} \ B_r^+\\
u_k=g_{\varepsilon_{k}} \quad\mathrm{or}\quad A_{\varepsilon_{k}}\nabla u_k\cdot\nu=h_{\varepsilon_k} &\mathrm{on} \ B'_r,
\end{cases}
\end{equation}
where $f_k,V_k,b_k$ are sequences of smooth functions which strongly converge in the right variable exponent Lebesgue spaces to $f,V,b$, while $A_{\varepsilon_k},F_{\varepsilon_k},g_{\varepsilon_k},h_{\varepsilon_k}$ are mollifications of $A,F,g,h$ as in \eqref{regucoe} and \eqref{regucoe1}.

%
%
Hence, solutions to \eqref{poiss2eps} satisfy the estimate
\begin{equation}\label{C1alphaesteps}
\|u_k\|_{C^{1,\alpha(\cdot)}(B^+_{r/2})}\leq c\left(\|u_k\|_{L^\infty(B^+_r)}+\|f_k\|_{L^{p(\cdot)}(B^+_r)}+\|F_{\varepsilon_k}\|_{C^{0,\alpha(\cdot)}(B^+_r)}+\begin{cases}\|g_{\varepsilon_k}\|_{C^{1,\alpha(\cdot)}(B'_r)}\quad\mathrm{or}\\
\|h_{\varepsilon_k}\|_{C^{0,\alpha(\cdot)}(B'_r)}\end{cases}\right),
\end{equation}
with a constant which depends on $\|V_k\|_{L^{m_1(\cdot)}(B_r^+)}$ and $\|b_k\|_{L^{m_2(\cdot)}(B_r^+)}$ and possibly is not uniform in $k$ (see for instance \cite{Amb,GiaMar,GilTru}). Now we are going to show that actually the constant in the estimate above can be taken uniform in $k$. The proof follows the contradiction argument developed in \cite{SoaTer10}.

A standard way to deal with inhomogeneous boundary data, in order to prove gradient estimates, is to extend them inside the domain to functions with the same regularity and consider the equation satisfied by the difference of the original solution and such extension. In this way, due to the linearity of the equation, one could pass to the study of a problem with homogeneous boundary data possibly with some new forcing terms appearing. We would like to remark here that, when boundary data are prescribed in variable exponent H\"older spaces, it is not even clear, at least for us, if an extension of such data inside the domain could be easily defined and how in that case the extension of the variable exponent itself should be defined. For this reason we will work directly on the problem with inhomogeneous boundary conditions.

\begin{remark}\label{alphastar}
We would like to remark here that actually we already know that solutions to \eqref{poiss2eps} enjoy a uniform $C^{1,\underline\alpha}$ local bound. This is due to classical results (see for instance \cite{GilHor,Lie1}, for the sharp requirement on the drift term we refer also to the more recent \cite{Dong1,Dong2}) once we have observed that variable exponent Lebesgue and H\"older spaces $L^{p(\cdot)}$ and $C^{k,\alpha(\cdot)}$ trivially embed into $L^{\underline p}$ and $C^{k,\underline\alpha}$. Hence, assuming $\|V_k\|_{L^{m_1(\cdot)}(B_r^+)}\leq c_1$ and $\|b_k\|_{L^{m_2(\cdot)}(B_r^+)}\leq c_2$, we already have a uniform constant such that
\begin{equation*}\label{2C1alphaesteps}
\|u_k\|_{C^{1,\underline\alpha}(B^+_{r/2})}\leq c\left(\|u_k\|_{L^\infty(B^+_r)}+\|f_k\|_{L^{p(\cdot)}(B^+_r)}+\|F_{\varepsilon_k}\|_{C^{0,\alpha(\cdot)}(B^+_r)}+\begin{cases}\|g_{\varepsilon_k}\|_{C^{1,\alpha(\cdot)}(B'_r)}\quad\mathrm{or}\\
\|h_{\varepsilon_k}\|_{C^{0,\alpha(\cdot)}(B'_r)}\end{cases}\right),
\end{equation*}
for the not sharp and possibly small $\underline\alpha\leq \min\{1-n/\underline p,1-n/\underline m_1,1-n/\underline m_2\}$. Of course, the constant is uniform if the variable coefficients of the equations are uniformly bounded in the $\underline\alpha$-H\"older space.
Anyway, the proof of this fact is actually contained in the following result by reasoning in two steps; that is, proving first the not sharp uniform bound with $\underline\alpha$ and then using this information to get the sharp result (see \cite[Remark 5.3]{SirTerVit1} for more details about this procedure).
\end{remark}

\begin{Proposition}\label{prop2}
As $k\to+\infty$ let $\{u_k\}$ be a family of solutions to \eqref{poiss2eps} in $B^+_r$ with $r\leq1/2$. Let $p,m_1,m_2\in\mathcal{P}^{\log}(B^+_r)$ with $\underline p,\underline m_1,\underline m_2>n$. Let $\|V_k\|_{L^{m_1(\cdot)}(B_r^+)}\leq c_1$ and $\|b_k\|_{L^{m_2(\cdot)}(B_r^+)}\leq c_2$. Let also $\alpha\in \mathcal{A}^{\log}(B^+_r)$ with \eqref{alpha2}. Then, there exists a positive constant independent from $k$ such that \eqref{C1alphaesteps} holds true.
\end{Proposition}
\proof
Without loss of generality we can assume the existence of a uniform constant $c>0$ as $k\to+\infty$ such that the terms in the right hand side in \eqref{C1alphaesteps} are uniformly bounded; that is,
\begin{equation*}
\|u_k\|_{L^\infty(B^+_r)}+\|f_k\|_{L^{p(\cdot)}(B^+_r)}+\|F_{\varepsilon_k}\|_{C^{0,\alpha(\cdot)}(B^+_r)}+\begin{cases}\|g_{\varepsilon_k}\|_{C^{1,\alpha(\cdot)}(B'_r)}\quad\mathrm{or}\\
\|h_{\varepsilon_k}\|_{C^{0,\alpha(\cdot)}(B'_r)}\end{cases}\leq c.
\end{equation*}
We have already assumed that 
\begin{equation*}
\|V_k\|_{L^{m_1(\cdot)}(B^+_r)}+\|b_k\|_{L^{m_2(\cdot)}(B^+_r)}\leq c.
\end{equation*}
{\bf Step 1: the contradiction argument.} We argue by contradiction; that is, there exists $\alpha\in \mathcal{A}^{\log}(B^+_r)$ with \eqref{alpha2}, and a subsequence of solutions (always denoted by $\{u_k\}$) to \eqref{poiss2eps} such that
$$\|\eta u_k\|_{C^{1,\alpha(\cdot)}(B^+_r)}\to+\infty,$$
where the function $\eta$ is a radial and decreasing cut-off function such that $\eta\in C^\infty_c(B_r)$ with $0\leq\eta\leq1$, $\eta\equiv1$ in $B_{r/2}$ and $\mathrm{supp}\eta=B=B_{3r/4}$. Moreover we take $\eta\in C^{0,1}(B)$ with $\partial_j\eta\in C^{0,1}(B)$ for any $j=1,...,n$, with the same constant $\ell$, that is $\eta(z)\leq \ell d(z,\partial B)$ and $\partial_j\eta(z)\leq \ell d(z,\partial B)$. We infer that the H\"older seminorm tends to infinity:
\begin{equation*}
\max_{j=1,...,n}\sup_{\substack{z,\zeta \in B^+_{r}\\z\neq\zeta}}\frac{\abs{\partial_j(\eta u_k)(z)-\partial_j(\eta u_k)(\zeta)}}{\abs{z-\zeta}^{\alpha(z)}}=L_k \to +\infty,
\end{equation*}
where $\partial_j=\partial_{x_j}$ for any $j=1,...,n$. Up to consider a subsequence, there exist $i\in\{1,...,n\}$, and two sequences of points $z_k,\zeta_k$ in $B\cap\{x_n\geq0\}$ such that
\begin{equation*}
\frac{\abs{\partial_i(\eta u_k)(z_k)-\partial_i(\eta u_k)(\zeta_k)}}{\abs{z_k-\zeta_k}^{\alpha(z_k)}}=L_k.
\end{equation*}
We remark that it is not possible that the H\"older seminorms of the sequence of derivatives $\partial_i(\eta u_k)$ stay bounded while their $L^\infty$-norms explode. This would be in contradiction with the uniform energy bound of the sequence $u_k$. We define $r_k=|z_k-\zeta_k|\leq2r\leq1$.

\noindent{\bf Step 2: the blow-up sequences.} Now we want to define two blow up sequences: let
\begin{equation*}
v_k(z)=\frac{\eta(\hat z_k+r_kz)}{L_kr_k^{1+\alpha(z_k)}}\left(u_k(\hat z_k+r_kz)-u_k(\hat z_k)\right),\qquad w_k(z)=\frac{\eta(\hat z_k)}{L_kr_k^{1+\alpha(z_k)}}\left(u_k(\hat z_k+r_kz)-u_k(\hat z_k)\right),
\end{equation*}
for $z\in B(k):=\frac{B^+_r-\hat z_k}{r_k}$ and $\hat z_k\in B\cap\{x_n\geq0\}$ to be determined. In any case $\hat z_k=(z_k',\hat z_{k,n})$ where $z_k=(z_k',z_{k,n})$. Let $B(\infty)=\lim_{k\to+\infty}B(k)$. We consider the two different cases as in the proof of Proposition \ref{prop1}; that is,
\begin{itemize}
\item[{\bf Case 1 :}] the term $\frac{z_{k,n}}{r_k}$ is unbounded, then we choose $\hat z_{k,n}=z_{k,n}$. In other words, $\hat z_k=z_k$;
\item[{\bf Case 2 :}] the term $\frac{z_{k,n}}{r_k}$ is bounded, then we choose $\hat z_{k,n}=0$. In other words, $\hat z_k=(z'_k,0)$.
\end{itemize}

In {\bf Case 1}, since points $z_k$ lie on a compact set, then we already know that $r_k\to0$. The fact that $r_k\to0$ in {\bf Case 2} has to be proved after suitably adjusting the blow-up sequences: this helps also to have the sequence of derivatives uniformly bounded in a point, in order to apply the Ascoli-Arzel\'a convergence theorem. Such adjustment consists in subtracting a linear term:
$$\overline v_k(z)=v_k(z)-\nabla v_k(0)\cdot z,\qquad\overline w_k(z)=w_k(z)-\nabla w_k(0)\cdot z.$$
One can see that
\begin{itemize}
\item[$i)$] $\overline v_k(0)=\overline w_k(0)=0$ and $|\nabla\overline v_k|(0)=|\nabla\overline w_k|(0)=0$;
\item[$ii)$] $[\partial_j\overline v_k]_{C^{0,\alpha(\cdot)}(K)}=[\partial_jv_k]_{C^{0,\alpha(\cdot)}(K)}$ for any compact $K\subset \overline{B(\infty)}$ and any $j=1,...,n$. Moreover, fixing any compact set $K\subset \overline{B(\infty)}$, there exists $\overline k$ such that for any $k>\overline k$,
\begin{equation*}\label{abovec19}
\sup_{\substack{z,\zeta \in K\\z\neq \zeta}}\frac{\abs{\partial_jv_k(z)-\partial_jv_k(\zeta)}}{\abs{z-\zeta}^{\alpha(\hat z_k+r_kz)}}\leq c(K);
\end{equation*}
Nevertheless, for the $i^{\mathrm{th}}$-partial derivatives
\begin{equation*}\label{abovec191}
1\leq\sup_{\substack{z,\zeta \in K\\z\neq \zeta}}\frac{\abs{\partial_iv_k(z)-\partial_iv_k(\zeta)}}{\abs{z-\zeta}^{\alpha(\hat z_k+r_kz)}}.
\end{equation*}
This is implied by the following
\begin{eqnarray}\label{nondeg}
&&\left|\partial_iv_k\left(\frac{z_k-\hat z_k}{r_k}\right)-\partial_iv_k\left(\frac{\zeta_k-\hat z_k}{r_k}\right)\right|\nonumber\\
&=&\left|\frac{1}{L_kr_k^{\alpha(z_k)}}(\partial_i(\eta u_k)(z_k)-\partial_i(\eta u_k)(\zeta_k))+\frac{u_k(\hat z_k)}{L_kr_k^{\alpha(z_k)}}(\partial_i\eta(z_k)-\partial_i\eta(\zeta_k))\right|\nonumber\\
&=&1+O\left(\frac{|u_k(\hat z_k)|r_k^{1-\alpha(z_k)}\ell}{L_k}\right)=1+o(1)=\left|\frac{z_k-\hat z_k}{r_k}-\frac{z_k-\hat z_k}{r_k}\right|^{\alpha(z_k)}+o(1);
\end{eqnarray}
\item[$iii)$] the sequences $\{\overline v_k\}, \{\overline w_k\}$ have the same asymptotic behaviour on compact subsets of $\overline{B(\infty)}$ and converge to the same function $\overline v$. This is done by equicontinuity and uniform boundedness of the family of functions $\overline v_k$ and $\nabla\overline v_k$ on compact subsets of $\overline{B(\infty)}$. The limit function $\overline v$ possesses gradient which is globally $\alpha_\infty$-H\"older continuous ($\alpha_\infty=\alpha(z_\infty)$) and is non constant passing to the limit in \eqref{nondeg} since
$$|\partial_i\overline v(0)-\partial_i\overline v(\overline z)|=1\qquad\mathrm{or}\qquad|\partial_i\overline v(\overline z_1)-\partial_i\overline v(\overline z_2)|=1,$$
respectively in {\bf Cases 1,2}, where $\overline z,\overline z_1, \overline z_2$ are the same as in the proof of Proposition \ref{prop1};
\item[$iv)$] a contradiction argument shows that $r_k\to0$ also in {\bf Case 2}: seeking a contradiction, let us suppose that $r_k\to\overline r>0$. Hence,
$$\sup_{z\in B(k)}|v_k(z)|\leq\frac{2\|\eta\|_{L^\infty(B_r)}\|u_k\|_{L^\infty(B^+_r)}}{r_k^{1+\alpha(z_k)}L_k}\leq\frac{c}{\overline r^{1+\alpha(z_\infty)}L_k}\to0,$$
which means that $v_k\to0$ uniformly on compact subsets of $\overline{B(\infty)}$. This fact implies also that pointwisely in $B(\infty)$
$$\overline v(z)=\lim_{k\to+\infty}\nabla v_k(0)\cdot z.$$
Since $0\in \overline{B(k)}$ for any $k$, it is easy to see that $B(\infty)$ contains a half ball $B^+_R$, for a small enough radius $R>0$. If the sequence $\{\partial_jv_k(0)\}$ was unbounded at least for $j=1,...,n$, then
$$|\overline v(Re_j)|=R\lim_{k\to+\infty}|\nabla v_k(0)\cdot e_j|=+\infty,$$
which is in contradiction with the fact that $\overline v\in C^{1,\alpha_\infty}(B^+_R)$ and hence bounded. Hence, $\{\nabla v_k(0)\}$ is a bounded sequence, and up to consider a subsequence, it converges to a vector $\nu\in\mathbb{R}^{n}$ and $\overline v(z)=\nu\cdot z$, which is in contradiction with the fact that $\overline v$ has non constant gradient.

Hence we can conclude that
\begin{equation*}
B(\infty)=\begin{cases}
\R^{n} &\mathrm{in \ {\bf Case \ 1}}\\
\{x_n>0\} &\mathrm{in \ {\bf Case \ 2}}.
\end{cases}
\end{equation*}
In fact, even if this time the blow-up points $z_k,\zeta_k$ may be on the boundary of the support of the cut-off function, we can ensure that
$$B_{\frac{r}{8r_k}}\subset \frac{B_{r}-\hat z_k}{r_k}.$$
\end{itemize}

\noindent{\bf Step 3: the rescaled equations.} Let $k>k_0$ large enough. Functions $\overline w_k$ solve in $B(k)$
\begin{eqnarray}\label{Lakoverw139}
-\mbox{div}\!\left(A_{\varepsilon_k}(\hat z_k)\nabla \overline w_k\right)(z)
&=&\frac{\eta(\hat z_k)}{L_k}r_k^{1-\alpha(z_k)}f_{k}(\hat z_k+r_kz)\nonumber\\
&&+\frac{\eta(\hat z_k)}{L_k}r_k^{-\alpha(z_k)}\mbox{div}\!\left(F_{\varepsilon_k}(\hat z_k+r_k\cdot)-F_{\varepsilon_k}(\hat z_k)\right)(z)\nonumber\\
&&+\frac{\eta(\hat z_k)}{L_k}r_k^{1-\alpha(z_k)}V_{k}(\hat z_k+r_kz)u_k(\hat z_k+r_kz)\nonumber\\
&&+r_kb_{k}(\hat z_k+r_kz)\cdot\nabla w_k(z)\nonumber\\
&&+\mbox{div}\!\left(\left(A_{\varepsilon_k}(\hat z_k+r_k\cdot)-A_{\varepsilon_k}(\hat z_k)\right)\nabla  w_k\right)(z),
\end{eqnarray}
plus possibly boundary conditions in {\bf Case 2} (we will deal with boundary terms later). The first and third terms vanish arguing as in \eqref{L1loc}. In order to make vanish the fourth and fifth terms we use Remark \ref{alphastar}. In fact, using Lemma \ref{keylemma}, which gives uniform-in-$k$ $\alpha(\cdot)$-H\"older continuity of the coefficients $a_{ij}$'s and using \eqref{cappucciononcappuccio} we have
\begin{eqnarray*}
&&\int_{B_R\cap B(\infty)}\left|\left(A_{\varepsilon_k}(\hat z_k+r_kz)-A_{\varepsilon_k}(\hat z_k)\right)\nabla  w_k\cdot\nabla\phi(z)\right|\\
&\leq&c\int_{B_R\cap B(\infty)}r_k^{\max\{\alpha(\hat z_k+r_kz),\alpha(\hat z_k)\}}|\nabla  w_k|\\
&\leq& \frac{c(R)}{L_k}\int_{B_R\cap B(\infty)}|\nabla  u_k(\hat z_k+r_kz)|,
\end{eqnarray*}
using the $\log$-H\"older continuity of $\alpha$ having
$$|\alpha(\hat z_k+r_kz)-\alpha(\hat z_k)|\cdot|\log r_k|\leq\omega(r_k|z|)|\log r_k|\leq\omega(r_kR)|\log r_k|\leq c.$$
Were $|\nabla u_k|$ uniformly bounded, we could promptly conclude. However, this information is given by the uniform $C^{1,\underline\alpha}$ estimate in Remark \ref{alphastar}.
%
Similarly,
\begin{eqnarray*}
&&r_k\left|\int_{B_R\cap B(\infty)}b_{k}(\hat z_k+r_kz)\cdot\nabla w_k(z)\phi(z)\mathrm{d}z\right|\nonumber\\
&\leq& c\frac{\eta(\hat z_k)r_k^{1-\alpha(z_k)}}{L_k}\int_{B_R\cap B(\infty)}|b_k(\hat z_k+r_kz)|\cdot|\nabla u_k(\hat z_k+r_kz)|\nonumber\\
&\leq&c\frac{r_k^{1-\alpha(\hat z_k)-\frac{n}{m_2(\hat z_k)}}}{L_k}\sup_{B_R\cap B(\infty)}|\nabla u_k(\hat z_k+r_kz)|\to0,
\end{eqnarray*}
using as before the information given by Remark \ref{alphastar}; that is, boundedness of the $|\nabla u_k|$'s.
In {\bf Case 2} we have to deal with boundary conditions. In case of Dirichlet boundary conditions, using a first order Taylor expansion of $g_{\varepsilon_k}$ with Lagrange form of the remainder, the $C^{1,\alpha(\cdot)}$ regularity of $g$ and Lemma \ref{keylemma} imply on compact sets $K'$ of $\{x_n=0\}$

\begin{eqnarray*}
|\overline w_k(x',0)|&=&|w_k(x',0)-\nabla w_k(0)\cdot(x',0)|\\
&=&\frac{\eta(\hat z_k)}{L_kr_k^{1+\alpha(z_k)}}\left|g_{\varepsilon_k}(\hat z_k+r_k(x',0))-g_{\varepsilon_k}(\hat z_k)-r_k\nabla_{x'}g_{\varepsilon_k}(\hat z_k)\cdot x'\right|\\
&\leq& c \frac{r_k^{1+\max\{\alpha(\hat z_k+r_k(\xi',0)),\alpha(\hat z_k)\}}}{L_kr_k^{1+\alpha(z_k)}}\to0.
\end{eqnarray*}
In the last line, the point $\xi'$ stands for an intermediate point between $0$ and $x'$.

In case of Neumann boundary condition, after integration by parts in the equation we can show that the contribution on the boundary vanishes

\begin{eqnarray*}
&&\left|\int_{B_R\cap\partial B(\infty)}(A_{\varepsilon_k}(\hat z_k+r_kz)\nabla w_k(z)-A_{\varepsilon_k}(\hat z_k)\nabla w_k(0))\cdot\nu\phi(z)\right|\\
&\leq&\frac{\eta(\hat z_k)r_k}{L_kr_k^{1+\alpha(z_k)}}\int_{B_R\cap\partial B(\infty)}|h_{\varepsilon_k}(\hat z_k+r_kz)-h_{\varepsilon_k}(\hat z_k)|\cdot|\phi(z)|\\
&\leq& c\frac{1}{L_kr_k^{\alpha(z_k)}}\int_{B_R\cap\partial B(\infty)}r_k^{\max\{\alpha(\hat z_k+r_kz),\alpha(\hat z_k)\}}\to0.
\end{eqnarray*}

The second term in the right hand side of \eqref{Lakoverw139}, which is the term with fields $F_{\varepsilon_k}$, can be treated in a similar way.

\noindent{\bf Step 4: the limit equation and a Liouville type theorem.}
One can reason as in the last part of proof of Proposition \ref{prop1}, obtaining that the limit function $\overline v$  belongs to $H^1_{\mathrm{loc}}(\mathbb{R}^{n})$ or $(H^1_{\{x_n=0\}})_{\mathrm{loc}}(\overline{\R^{n}_+})/H^1_{\mathrm{loc}}(\overline{\R^{n}_+})$ and solves
\begin{equation*}\label{entire02}
-\mathrm{div}\left(\hat A\nabla \overline v\right)=0 \quad \mathrm{in} \ \R^{n},\qquad \mathrm{or}\qquad
\begin{cases}
-\mathrm{div}\left(\hat A\nabla \overline v\right)=0 & \mathrm{in} \ \R^{n}_+,\\
\overline v=0 \qquad\mathrm{or}\qquad \hat A\nabla \overline v\cdot\nu=0 &\mathrm{on} \ \{x_n=0\},
\end{cases}
\end{equation*}
respectively in {\bf Cases 1,2} with constant coefficients limit matrix $\hat A=A(z_\infty)=\lim A_{\varepsilon_k}(\hat z_k)$. Nevertheless, $\overline v$ is non constant and has non constant gradient which is globally $\alpha_\infty$-H\"older continuous in $\R^n$ or $\R^n_+$.
After applying the linear map introduced in Proposition \ref{prop1}, we can reabsorb the constant coefficients limit matrix, and after possibly reflecting with respect to the hyperplane $C^{-1}\{x_n=0\}$, we have an entire harmonic function in $H^1_{\mathrm{loc}}(\mathbb{R}^{n})$ globally $C^{1,\alpha_\infty}$ with a non constant partial derivative. Hence this derivative is also harmonic and globally $\alpha_\infty$-H\"older continuous with $\alpha_\infty<1$, in clear contradiction with the Liouville theorem \cite[Corollary 2.3]{NorTavTerVer}.
\endproof

\section{An approximation scheme and Schauder estimates}
In this section we are going to construct an approximation scheme which allows to extend the estimates proved in the previous section to any weak solution to \eqref{eq1}; in other words, we are going to prove Theorem \ref{teo1} and Theorem \ref{teo2}. Then, in the last part of the section we show how to iterate the gradient estimate obtained in order to produce Schauder regularity for partial derivatives of our solutions of any order; that is, Corollary \ref{cor3}.

\subsection{An approximation scheme}
The approximation scheme exposed in this section follows some ideas contained in \cite[Section 6]{SirTerVit1}.
\begin{remark}\label{remarkcoe}
We remark that we can always localize our problem in half balls (or balls, which is a simpler case which we are not going to treat) with radii small enough to ensure local coercivity of the bilinear form given by
\begin{equation}\label{bilinear}
b(u,\phi)=\int_{B_r^+}A\nabla u\cdot\nabla\phi-Vu\phi-b\cdot\nabla u\phi,
\end{equation}
which implies, imposing boundary data on the curved part of the boundary $\partial^+B_r^+$, existence and uniqueness of solutions to \eqref{eq1}. Then, local estimates in generic domains are consequently obtained through a covering argument. In other words, our balls will be taken with $r\leq\min\{1/2,r_0\}$ where $r_0$ depends on the norms $\|V\|_{L^{m_1(\cdot)}}$ and $\|b\|_{L^{m_2(\cdot)}}$ and ellipticity constants of $A$ in the original domain where we would like to ensure regularity estimates. The coercivity can be ensured whenever $\underline m_1\geq\frac{n}{2}$ and $\underline m_2\geq n$. Since we are going to perturb the problem in order to regularize it, we would like to stress the fact that this coercivity radius $r_0$ can be taken in such a way as to guarantee coercivity of the $k$-perturbed bilinear forms, for $k$ large, as well. This is due to closeness of the regularized potentials and drifts to the original ones in Lebesgue norms, and closeness of regularized coefficients to original ones in $L^\infty$.
\end{remark}

\begin{proof}[Proof of Theorem \ref{teo1} and Theorem \ref{teo2}]
We develop a strategy which has some relevant differences in case of Dirichlet and Neumann boundary conditions at $B_r'$. We work in the setting of the regularization for H\"older estimates, but the reasoning applies a fortiori also in the case of regularization for gradient estimates. Taken a weak solution to \eqref{eq1} with either Dirichlet or Neumann boundary data at the flat boundary, as we have already noticed in Remark \ref{remarkcoe}, without loss of generality we can restrict ourselves to small radii $r\leq\min\{1/2,r_0\}$.

\noindent{\bf Case 1: Dirichlet BC.} Let $u\in H^1(B_r^+)$ be a generic weak solution to \eqref{eq1} with Dirichlet boundary condition at $B_r'$ and let us fix a radius $0<\overline r<r$. Along a sequence $k\to+\infty$ let us consider the regularizations of $A,f,F,V,b,g$ as in Section \ref{sectrego}. Mollifications $A_{\varepsilon}$ and $g_\varepsilon$ are well defined respectively in $B_{\frac{r+\overline r}{2}}^+$ and $B'_{\frac{r+\overline r}{2}}$ whenever $\varepsilon\leq\overline\varepsilon$ depending on $\overline r$. We remark that the trace $\Phi$ of the solution $u$ belongs to $H^{1/2}(\partial B_r^+)$ (since $B_r^+$ is a Lipschitz domain \cite{Gag}). Nevertheless, on $B_r'$ we know that $\Phi=g\in C^{0,\alpha(\cdot)}$. Hence, let us define
\begin{equation}\label{Phik}
\Phi_k:=\begin{cases}
\Phi &\mathrm{in \ }\partial^+B^+_{r}\\
\eta g_{\varepsilon_k}+(1-\eta)g &\mathrm{in \ }B'_{r}
\end{cases}
\end{equation}
and in general, for any regularization $G_\varepsilon$ of a datum $G$ by convolution, let us define $\tilde G_k:=\eta G_{\varepsilon_k}+(1-\eta)G$ in $B_r^+$,
where $\eta\in C^\infty_c(B_{\frac{r+\overline r}{2}})$ is a radially non increasing cut-off function with $0\leq\eta\leq1$ and $\eta\equiv1$ in $B_{\overline r}$. 
These junctions are introduced merely to ensure solvability for the Dirichlet problem. In fact, this way we will be able to prescribe traces, for the perturbed problems, belonging to $H^{1/2}(\partial B_r^+)$. It is very easy to see that $\|\Phi_k\|_{L^2(\partial B_r^+)}\leq c$ uniformly in $k$ and that $\Phi_k\to\Phi$ strongly in $L^2(\partial B_r^+)$. Actually, one can show that the fractional Gagliardo seminorm
\begin{equation}\label{H1/2}
[\Phi_k]_{H^{1/2}(\partial B_r^+)}=\iint_{\partial B_r^+\times\partial B_r^+}\frac{|\Phi_k(x)-\Phi_k(y)|^2}{|x-y|^{n}}\mathrm{d}x\mathrm{d}y\leq c
\end{equation}
is bounded uniformly in $k$ (we prove the previous claim in Remark \ref{claim}). Let us consider, for any trace $\Phi_k$ an $H^1(B_r^+)$-extension $u_{\Phi_k}$. Hence
$$\|u_{\Phi_k}\|_{H^1(B_r^+)}\leq c\|\Phi_k\|_{H^{1/2}(\partial B_r^+)}\leq c$$
and hence $u_{\Phi_k}$ converge weakly, up to consider a subsequence, to a function $\overline v\in H^1(B_r^+)$. Nevertheless, by compact trace embedding, $\mathrm{Tr}u_{\Phi_k}\to \mathrm{Tr}\overline v$ strongly in $L^2(\partial B_r^+)$ which implies that $\mathrm{Tr}\overline v=\Phi$. Hence, as extension of $\Phi$ we consider exactly $u_\Phi:=\overline v$ and we define $w=u-u_\Phi\in H^1_0(B_r^+)$. In particular, $w$ is the unique solution to
\begin{equation}\label{poiss3eps111}
\begin{cases}
-\mathrm{div}\left(A\nabla w\right)=f+\mathrm{div}F+Vw+b\cdot\nabla w+\mathrm{div}\left(A\nabla u_\Phi\right)+Vu_\Phi+b\cdot\nabla u_\Phi &\mathrm{in} \ B_r^+\\
w=0 &\mathrm{on} \ \partial B^+_r.
\end{cases}
\end{equation}
This is due to continuity and coercivity of the bilinear form \eqref{bilinear} and applying the Lax-Milgram theorem with linear form
$$\langle L,\phi\rangle=\int_{B_r^+}f\phi+F\cdot\nabla\phi+A\nabla u_\Phi\cdot\nabla\phi+Vu_\Phi\phi+b\cdot\nabla u_\Phi\phi.$$

Let us now consider the unique solution $w_k\in H^1_0(B_r^+)$ to
\begin{equation}\label{poiss3eps11}
\begin{cases}
-\mathrm{div}\left(\tilde A_{k}\nabla w_k\right)=f_k+\mathrm{div}F_{k}+V_kw_k+b_k\cdot\nabla w_k +\mathrm{div}\left(\tilde A_k\nabla u_{\Phi_k}\right)+V_ku_{\Phi_k}+b_k\cdot\nabla u_{\Phi_k}&\mathrm{in} \ B_r^+\\
w_k=0 &\mathrm{on} \ \partial B^+_r.
\end{cases}
\end{equation}
As we have previously remarked, solvability and uniqueness of the previous problem in $H^1_0(B_r^+)$ come from coercivity of the $k$-perturbed bilinear form
\begin{equation}\label{bilineark}
b_k(v,\phi)=\int_{B_r^+}\tilde A_k\nabla v\cdot\nabla\phi-V_kv\phi-b_k\cdot\nabla v\phi,
\end{equation}
and applying the Lax-Milgram theorem with linear form
$$\langle L_k,\phi\rangle=\int_{B_r^+}f_k\phi+F_k\cdot\nabla\phi+\tilde A_k\nabla u_{\Phi_k}\cdot\nabla\phi+V_ku_{\Phi_k}\phi+b_k\cdot\nabla u_{\Phi_k}\phi.$$
The sequence of the $w_k$'s converges to $w$.
This is done using strong convergences of sequences $\{f_k\}$, $\{F_k\}$, $\{V_k\}$ and $\{b_k\}$ to $f,F,V,b$ in the right variable exponent Lebesgue spaces. In fact, testing \eqref{poiss3eps11} with $w_k$ and using Sobolev embeddings one gets a uniform-in-$k$ constant such that
$$\|w_k\|_{H^1_0(B_r^+)}\leq c\left(\|f_k\|_{L^{p(\cdot)}(B_r^+)}+\|F_k\|_{L^{q(\cdot)}(B_r^+)}+\|\Phi_k\|_{H^{1/2}(\partial B_r^+)}\right).$$
The constant above is uniform since we have $\|V_k\|_{L^{m_1(\cdot)}(B_r^+)}+\|b_k\|_{L^{m_2(\cdot)}(B_r^+)}\leq c$ (the constant itself depends on them). We have used also the fact that there exists a positive constant which is uniform in $k$ and depends only on ellipticity constants of $A$ such that the quadratic forms associated to $\tilde A_k$ satisfy
\begin{equation*}\label{quadratic}
\frac{1}{c}\int_{B_r^+}|\nabla v|^2\leq \int_{B_r^+}\tilde A_k\nabla v\cdot\nabla v\leq c\int_{B_r^+}|\nabla v|^2.
\end{equation*}
Hence, one gets easily weak convergence in $H^1_0(B_r^+)$ to a certain function. Using pointwise convergences and Vitali's convergence theorem we obtain that such a function is solution to \eqref{poiss3eps111} and hence must coincide with $w$ by uniqueness (for details we refer to \cite[Lemma 2.12]{SirTerVit1}).

Hence, $u_k:=w_k+u_{\Phi_k}$ weakly converges to $w+u_{\Phi}=u$ with
$$\|u_k\|_{H^1(B_r^+)}\leq\|w_k\|_{H^1_0(B_r^+)}+c\|\Phi_k\|_{H^{1/2}(\partial B_r^+)}$$
and it is solution to

\begin{equation*}
\begin{cases}
-\mathrm{div}\left(\tilde A_{k}\nabla u_k\right)=f_k+\mathrm{div}F_{k}+V_ku_k+b_k\cdot\nabla u_k  &\mathrm{in} \ B_r^+\\
u_k=\Phi_k &\mathrm{on} \ \partial B^+_r.
\end{cases}
\end{equation*}

We remark that sequence $\{u_k\}$ satisfies the uniform estimate \eqref{C0alphaesteps} in Proposition \ref{prop1} on $B_{\overline r}^+$; that is,
\begin{equation*}
\|u_k\|_{C^{0,\alpha(\cdot)}(B_{\overline r/2}^+)}\leq c\left(\|u_k\|_{L^\infty(B_{\overline r}^+)}+\|f_k\|_{L^{p(\cdot)}(B_{\overline r}^+)}+\|F_k\|_{L^{q(\cdot)}(B_{\overline r}^+)}+\|g_{\varepsilon_k}\|_{C^{0,\alpha(\cdot)}(B'_{\overline r})}\right),
\end{equation*}
with a constant which depends on $\|V_k\|_{L^{m_1(\cdot)}}$ and $\|b_k\|_{L^{m_2(\cdot)}}$. Since the right hand side is in turns uniformly bounded in $k$ (by convergences, uniform energy bounds and Lemma \ref{keylemma}), in particular the convergence $u_k\to u$ is uniform on compact subsets of $B^+_{\overline r/2}\cup B'_{\overline r/2}$, which means that one can pass to the limit in the previous estimate getting for the general solution $u$ the desired estimate (up to restrict a bit the radius of the ball $\tilde r<\overline r$).

\noindent{\bf Case 2: Neumann BC.}
Let us consider, along a sequence $k\to+\infty$, the unique solution $w_k$ to the following problem
\begin{equation*}\label{poiss3eps}
\begin{cases}
-\mathrm{div}\left(\tilde A_{k}\nabla w_k\right)=f_k+\mathrm{div}F_{k}+V_kw_k+b_k\cdot\nabla w_k+\mathrm{div}\left(\tilde A_k\nabla u\right)+V_ku+b_k\cdot\nabla u &\mathrm{in} \ B_r^+\\
\tilde A_{k}\nabla w_k\cdot\nu=h_k-\tilde A_k\nabla u\cdot\nu &\mathrm{on} \ B'_r\\
w_k=0 &\mathrm{on} \ \partial^+B^+_r.
\end{cases}
\end{equation*}
This time, solvability and uniqueness of the previous problem in $H^1_0(B_r^+\cup B_r')$ come from coercivity of the bilinear form \eqref{bilineark}
and applying the Lax-Milgram theorem with linear form
$$\langle L_k,\phi\rangle=\int_{B_r^+}f_k\phi+F_k\cdot\nabla\phi+\tilde A_k\nabla u\cdot\nabla\phi+V_ku\phi+b_k\cdot\nabla u\phi \ +\int_{B_r'}h_k\phi-\tilde A_k\nabla u\cdot\nu.$$
We argue as before obtaining easily that $w_k$ weakly converge to a function $w$ which is the unique solution to
\begin{equation*}\label{poiss4eps}
\begin{cases}
-\mathrm{div}\left(A\nabla w\right)=Vw+b\cdot\nabla w &\mathrm{in} \ B_r^+\\
A\nabla w\cdot\nu=0 &\mathrm{on} \ B'_r\\
w=0 &\mathrm{on} \ \partial^+B^+_r,
\end{cases}
\end{equation*}
which however is the zero function.

Hence, $u_k:=w_k+u$ weakly converges to $u$. Moreover, the $u_k$'s weakly solve
\begin{equation*}
\begin{cases}
-\mathrm{div}\left(\tilde A_{k}\nabla u_k\right)=f_k+\mathrm{div}F_{k}+V_ku_k+b_k\cdot\nabla u_k  &\mathrm{in} \ B_r^+\\
\tilde A_k\nabla u_k\cdot\nu=h_k &\mathrm{on} \  B'_r\\
u_k=\Phi &\mathrm{on} \ \partial^+ B^+_r.
\end{cases}
\end{equation*}

Of course, as in {\bf Case 1}, the sequence of the $u_k$'s enjoys the uniform estimate \eqref{C0alphaesteps} in Proposition \ref{prop1}. Hence $u$ inherits the regularity estimate by uniform convergence on compact sets (up to consider a smaller ball with radius $\tilde r<\overline r$).

The case of gradient estimates is very similar (which is the proof of Theorem \ref{teo2}). We only remark that at some point, in order to transfer the uniform estimate \eqref{C1alphaesteps} of Proposition \ref{prop2} to the limit, one has to use the uniform convergence on compact sets of the sequences of partial derivatives $\{\partial_i u_k\}$.
\end{proof}

\begin{remark}\label{claim}
Let $\Phi_k$ as in \eqref{Phik}. Then \eqref{H1/2} holds true.
\end{remark}
\proof
We are going to estimate uniformly in $k$ the following fractional Gagliardo seminorm
$$[\Phi_k]_{H^{1/2}(\partial B_r^+)}=\iint_{\partial B_r^+\times\partial B_r^+}\frac{|\Phi_k(x)-\Phi_k(y)|^2}{|x-y|^{n}}.$$
One can split the double integral into several pieces, using different information to manage the estimate. We can reduce ourselves to the following case:

\noindent{\bf Case 1:} $x\in B'_{\frac{r+\overline r}{2}}$ and $y\in\partial B_r^+$.

In fact, the complementary can be divided into the union of $x\in\partial B_r^+\setminus B'_{\frac{r+\overline r}{2}}$ and $y\in\partial B^+_r\setminus B'_{\frac{r+\overline r}{2}}$ with $x\in\partial B_r^+\setminus B'_{\frac{r+\overline r}{2}}$ and $y\in B'_{\frac{r+\overline r}{2}}$. By symmetry, the third case is part of the first one, while the integral contribution in the second case can be trivially estimated by
$$\iint_{(\partial B_r^+\setminus B'_{\frac{r+\overline r}{2}})^2}\frac{|\Phi_k(x)-\Phi_k(y)|^2}{|x-y|^{n}}=\iint_{(\partial B_r^+\setminus B'_{\frac{r+\overline r}{2}})^2}\frac{|\Phi(x)-\Phi(y)|^2}{|x-y|^{n}}\leq[\Phi]_{H^{1/2}(\partial B_r^+)}.$$

Hence, we consider the subcase

\noindent{\bf Case 1.1:} $x\in B'_{\frac{r+\overline r}{2}}$ and $y\in\partial^+ B_r^+$.

In this case points are not too close; that is, $|x-y|\geq\frac{r-\overline r}{2}$, and hence the integral contribution can be estimated by some constant times $\|\Phi_k\|_{L^2(\partial B_r^+)}$, which is however uniformly bounded.

Then, we consider another subcase

\noindent{\bf Case 1.2:} $x\in B'_{\frac{r+\overline r}{2}}$ and $y\in B'_{r}\setminus B'_{\frac{r+\overline r}{2}}$.

Actually, if $x\in B'_{\overline r}$, again points are not so close $|x-y|\geq \frac{r-\overline r}{2}$. Hence, if $x\in B'_{\frac{r+\overline r}{2}}\setminus B'_{\overline r}$ and $y\in B'_{r}\setminus B'_{\frac{r+\overline r}{2}}$, using that $\eta(y)=0$, $\underline\alpha$-H\"older continuity of $g$ and Lipschitz continuity of $\eta$
\begin{eqnarray*}
|\Phi_k(x)-\Phi_k(y)|^2&=&|\eta(x)g_{\varepsilon_k}(x)+(1-\eta(x))g(x)-g(y)|^2\\
&=&|(\eta(x)-\eta(y))(g_{\varepsilon_k}(x)-g(x))+g(x)-g(y)|^2\\
&\leq& c\varepsilon^{2\underline\alpha}|x-y|^{2}\left(\int_{\mathrm{supp}\tilde\eta}\tilde\eta(t)|t|^{\underline\alpha}\right)^2+|g(x)-g(y)|^2
\end{eqnarray*}
which divided by the kernel $|x-y|^n$ are integrable terms.

Hence, we consider the last subcase

\noindent{\bf Case 1.3:} $x\in B'_{\frac{r+\overline r}{2}}$ and $y\in B'_{\frac{r+\overline r}{2}}$.

We consider three subcases of {\bf Case 1.3}:

\noindent{\bf Case 1.3.1:} $x\in B'_{\frac{r+\overline r}{2}}\setminus B'_{\overline r}$ and $y\in B'_{\overline r}$.

Using that $1-\eta(y)=0$,
\begin{eqnarray*}
|\Phi_k(x)-\Phi_k(y)|^2&=&|\eta(x)g_{\varepsilon_k}(x)+(1-\eta(x))g(x)-g_{\varepsilon_k}(y)|^2\\
&=&|(\eta(x)-\eta(y))(g(x)-g_{\varepsilon_k}(x))+g_{\varepsilon_k}(x)-g_{\varepsilon_k}(y)|^2
\end{eqnarray*}
and hence, we can conclude if we are able to bound uniformly
\begin{equation}\label{Ik}
I_k:=\iint_{(B'_{\frac{r+\overline r}{2}})^2}\frac{|g_{\varepsilon_k}(x)-g_{\varepsilon_k}(y)|^2}{|x-y|^n}\leq c.
\end{equation}
One can show that also in the last two subcases we have left behind; that is,

\noindent{\bf Case 1.3.2:} $x\in B'_{\frac{r+\overline r}{2}}\setminus B'_{\overline r}$ and $y\in B'_{\frac{r+\overline r}{2}}\setminus B'_{\overline r}$

and

\noindent{\bf Case 1.3.3:} $x\in B'_{\overline r}$ and $y\in B'_{\overline r}$

the key point is the uniform bound of $I_k$ in \eqref{Ik}. Hence, in order to prove it, one can reason as in \cite[Lemma A.1.]{Bra},

\begin{eqnarray*}
I_k&=&\iint_{(x,y)\in(B'_{\frac{r+\overline r}{2}})^2}\frac{|\int_{t\in B'_{\varepsilon_k}}\tilde \eta_{\varepsilon_k}(t)(g(x-t)-g(y-t))|^2}{|x-y|^n}\\
&\leq&\iint_{(x,y)\in(B'_{\frac{r+\overline r}{2}})^2}\frac{\int_{t\in B'_{\varepsilon_k}}\tilde \eta_{\varepsilon_k}(t)|g(x-t)-g(y-t)|^2}{|x-y|^n},
\end{eqnarray*}
where in the last inequality we have used Jensen's inequality and the fact that $\tilde \eta_{\varepsilon_k}(t)\mathrm{d}t$ is a probability measure. We can conclude using the fact that $\tilde \eta_{\varepsilon_k}$ is supported in $B'_{\varepsilon_k}$ and hence for the function
$$H(x,y)=\frac{g(x)-g(y)}{|x-y|^{n/2}}\in L^2(B_r'\times B_r')$$
we have by continuity of translations
$$\|H(\cdot-t,\cdot-t)\|_{L^2(B'_{\frac{r+\overline r}{2}}\times B'_{\frac{r+\overline r}{2}})}\leq c,$$
which is uniform if $|t|\leq\varepsilon_k\leq\overline\varepsilon$.
\endproof

\subsection{Schauder estimates with variable exponent}
In this last section we prove Corollary \ref{cor3}; that is, we obtain local boundary Schauder estimates for weak solutions to

\begin{equation*}\label{eq2bis}
\begin{cases}
-\mathrm{div}\left(A\nabla u\right)=\mathrm{div}F &\mathrm{in \ } B_r\cap\Omega\\
u=g\quad\mathrm{or}\quad A\nabla u\cdot\nu=h &\mathrm{on \ } B_r\cap\partial\Omega,
\end{cases}
\end{equation*}
By applying the diffeomorphism introduced in \eqref{diffeomorphism}, and using the lemmas and reasonings in Section \ref{sectdiffeo}, the proof of Corollary \ref{cor3} is actually implied by the following result

\begin{Proposition}
Let $r>0$, $k\geq0$, $\alpha\in \mathcal{A}^{\log}(B^+_r)$, and $u$ be a weak solution to
\begin{equation}\label{eqscha}
\begin{cases}
-\mathrm{div}\left(A\nabla u\right)=\mathrm{div}F &\mathrm{in \ } B_r^+\\
u=g\quad\mathrm{or}\quad A\nabla u\cdot\nu=h &\mathrm{on \ } B_r',
\end{cases}
\end{equation}
with $A,F\in C^{k,\alpha(\cdot)}(B^+_r)$, $g\in C^{k+1,\alpha(\cdot)}(B'_r)$ or $h\in C^{k,\alpha(\cdot)}(B'_r)$. Then $u\in C^{k+1,\alpha(\cdot)}_{\mathrm{loc}}(B^+_r\cup B'_r)$.
\end{Proposition}

We would like to remark here that the proof is based on the iteration of the gradient estimate in Theorem \ref{teo2}.

\proof
We procede by induction. The result for $k=0$ is contained in Theorem \ref{teo2}. Hence, we suppose the result true for a generic integer $k\geq0$ and we prove it for the case $k+1$; that is, we are assuming
$A,F\in C^{k+1,\alpha(\cdot)}(B^+_r)$, $g\in C^{k+2,\alpha(\cdot)}(B'_r)$ or $h\in C^{k+1,\alpha(\cdot)}(B'_r)$, and we want to prove that actually $u\in C^{k+2,\alpha(\cdot)}_{\mathrm{loc}}(B^+_r\cup B'_r)$. First, we prove that actually
\begin{equation}\label{partiali}
\partial_{x_i}u\in C^{k+1,\alpha(\cdot)}_{\mathrm{loc}}(B^+_r\cup B'_r)\qquad \mathrm{for \  any \ } i=1,...,n-1.
\end{equation}
In fact, differentiating the equation, one gets
\begin{equation*}
\begin{cases}
-\mathrm{div}\left(A\nabla (\partial_{x_i}u)\right)=\mathrm{div}\left(\partial_{x_i}A\nabla u+\partial_{x_i}F\right) &\mathrm{in \ } B_r^+\\
\partial_{x_i}u=\partial_{x_i}g\quad\mathrm{or}\quad A\nabla (\partial_{x_i}u)\cdot\nu=\partial_{x_i}A\nabla u\cdot\nu+\partial_{x_i}h &\mathrm{on \ } B_r'.
\end{cases}
\end{equation*}
Using the inductive hypothesis, and the fact that we already know that $u\in C^{k+1,\alpha(\cdot)}_{\mathrm{loc}}(B^+_r\cup B'_r)$ (which implies $\nabla u\in C^{k,\alpha(\cdot)}_{\mathrm{loc}}(B^+_r\cup B'_r)$), then \eqref{partiali} follows.
In order to conclude, it remains to prove
\begin{equation}\label{partiali2}
\partial_{x_n}u\in C^{k+1,\alpha(\cdot)}_{\mathrm{loc}}(B^+_r\cup B'_r).
\end{equation}
Actually, \eqref{partiali2} follows if we prove that $\partial^2_{x_nx_n}u\in C^{k,\alpha(\cdot)}_{\mathrm{loc}}(B^+_r\cup B'_r)$. In fact, by \eqref{partiali} we already know that $\partial_{x_i}\partial_{x_n}u\in C^{k,\alpha(\cdot)}_{\mathrm{loc}}(B^+_r\cup B'_r)$ for any $i=1,...,n-1$. Hence, considering equation \eqref{eqscha}, we have the following
\begin{equation*}
-\partial_{x_n}\left(A\nabla u\cdot \vec e_n\right)=\mathrm{div}F+\sum_{i=1}^{n-1}\partial_{x_i}\left(A\nabla u\cdot \vec e_i\right).
\end{equation*}
Hence
\begin{equation*}
-a_{n,n}\partial^2_{x_nx_n}u=\mathrm{div}F+\sum_{i=1}^{n-1}\partial_{x_i}\left(A\nabla u\cdot \vec e_i\right)+\partial_{x_n}a_{n,n}\partial_{x_n}u+\partial_{x_n}\left(\sum_{i=1}^{n-1}a_{n,i}\partial_{x_i}u\right).
\end{equation*}
The uniform ellipticity of $A$ implies that $a_{n,n}(x)=A(x)\vec e_n\cdot \vec e_n\geq\lambda>0$, and together with the expression above allows to conclude. In fact $\partial^2_{x_nx_n}u$ can be expressed as sum of $C^{k,\alpha(\cdot)}_{\mathrm{loc}}(B^+_r\cup B'_r)$-functions.
\endproof

\section*{Acknowledgment}
I would like to thank Susanna Terracini for many fruitful conversations on the paper.

\end{document}